\documentclass{article}

\usepackage{amsmath}
\usepackage{amssymb}
\usepackage{amsthm}

\newtheorem{theorem}{Theorem}
\newtheorem{proposition}[theorem]{Proposition}
\newtheorem{corollary}[theorem]{Corollary}
\newtheorem{lemma}[theorem]{Lemma}
\theoremstyle{definition}
\newtheorem{definition}[theorem]{Definition}
\newtheorem{convention}[theorem]{Convention}
\newtheorem{example}[theorem]{Example}
\newtheorem{remark}[theorem]{Remark}

\usepackage{color}

\newcommand{\mcap}{{\cap}}
\newcommand{\mcup}{{\cup}}
\newcommand{\mwedge}{{\curlywedge}}
\newcommand{\mvee}{{\curlyvee}}
\newcommand{\mmid}{{\,|\,}}

\begin{document}

\title{Handlebody-knot invariants derived from unimodular Hopf algebras}
\author{Atsushi Ishii%
\thanks{The first author was supported by JSPS KAKENHI Grant Number 24740037.}
and Akira Masuoka%
\thanks{The second author was supported by JSPS KAKENHI Grant Number 23540039.}}
\date{}

\maketitle

\begin{abstract}
A handlebody-knot is a handlebody embedded in the $3$-sphere.
We establish a uniform method to construct invariants for handlebody-links.
We introduce the category $\mathcal{T}$ of handlebody-tangles and present it by generators and relations.
The result tells us that every functor on $\mathcal{T}$ that gives rise to invariants is derived from
what we call a quantum-commutative quantum-symmetric algebra in the target category.
The example of such algebras of our main concern is finite-dimensional unimodular Hopf algebras.
We investigate how those Hopf algebras give rise to handlebody-knot invariants.
\end{abstract}

\section{Introduction}

A \textit{handlebody-knot} is a handlebody embedded in the $3$-sphere $S^3$;
it is alternatively called a knotted handlebody or a spatial handlebody.
A \textit{handlebody-link} is a disjoint union of handlebodies embedded in $S^3$.
Two handlebody-links are \textit{equivalent} if there exists an isotopy of $S^3$
which takes one to the other, or equivalently if there exists
an orientation-preserving self-homeomorphism of $S^3$ which sends one to the other.
The aim of this paper is to establish a uniform method to construct invariants for handlebody-links.

A handlebody-knot was first introduced as a neighborhood equivalence class
of a spatial graph by Suzuki~\cite{Suzuki70}.
Since neighborhood equivalence classes of knots coincide with ambient isotopy classes of knots,
genus $1$ handlebody-knots correspond to knots,
which means that a handlebody-knot is a generalization of a knot.
Study of knots with invariants has made great progress
since the discovery of the Jones polynomial and the subsequent so-called quantum invariants;
see, for example,~\cite{Ohtsuki02}.
Quantum invariants are derived from representations of quantum groups.
Superiority of a quantum invariant is that it can be derived from any representation of any quantum group.
A functor from the category of tangles to that of vector spaces,
which is obtained via a representation of a quantum group,
gives a quantum invariant of tangles, especially of links.

Invariants of handlebody-links can be realized as those of spatial trivalent graphs
which are invariant under IH-moves~\cite{Ishii08},
where an IH-move is a local move on spatial trivalent graphs.
When a handlebody-link $H$ is a regular neighborhood of a spatial graph $K$,
we say that $H$ is represented by $K$.
In this paper trivalent graphs may contain circle components.
Then any handlebody-link can be represented by some spatial trivalent graph.
The existence of trivalent vertices
distinguishes mostly handlebody-links from ordinary links, and gives us the biggest barrier
when we construct functors on handlebody-tangles.
In order to get over this barrier,
we introduce quantum-commutative quantum-symmetric algebras (see Definition~\ref{def:q-symmetric-alg}),
and assign the multiplication mapping to a trivalent vertex.
Then invariance under IH-moves follows from the associativity of multiplication.
We obtain an invariant for handlebody-links with every quantum-commutative quantum-symmetric algebra.
A good example of quantum-commutative quantum-symmetric algebras is given by
finite-dimensional unimodular Hopf algebras, which include finite groups as the simplest example.
Our invariant derived from a finite group coincides with the number of the homomorphisms
from the fundamental group of the exterior of a handlebody-knot to the group.

We remark that Mizusawa and Murakami~\cite{MizusawaMurakami14} constructed quantum $U_q(sl_2)$
type invariants for handlebody-knots in $S^3$ via Yokota's invariants~\cite{Yokota96}.

This paper is organized as follows.
In Section 2, we introduce the category $\mathcal{T}$ of handlebody-tangles
and present it by generators and relations;
the result enables us to construct on $\mathcal{T}$ the functors which gives our invariants.
In Section 3, we introduce the notion of quantum-commutative quantum-symmetric algebras
and see how the invariants are obtained from those algebras.
In Section 4, we focus on quantum-commutative quantum-symmetric algebras
in the categories of Yetter--Drinfeld modules and show that
every finite-dimensional unimodular Hopf algebra $A$ is a quantum-commutative quantum-symmetric algebra
in the category of Yetter--Drinfeld modules over $A$.
In Section 5, we investigate the invariants derived from unimodular Hopf algebras
in some algebraic and geometric situations such as the disk sum or the mirror image.
In Sections 6 and 7, we give examples of invariants derived from unimodular Hopf algebras
together with their data needed to compute the invariants.

\section{The category of handlebody-tangles}

A \textit{handlebody-tangle} is a disjoint union of handlebodies
embedded in a cube $I^3$ such that the intersection of the handlebodies
and the boundary of $I^3$ is the union of sequences of disks in the
top and bottom squares as shown in Figure~\ref{fig:handlebody-tangle}.
We call the disks in the top (resp.~bottom) square the
\textit{top (resp.~bottom) end disks} of the handlebody-tangle.
Two handlebody-tangles are assumed to be the same if one can be
transformed into the other by an isotopy of $I^3$ preserving the order
of the end disks in the top and bottom squares.
We remark that a handlebody-tangle with no end disks corresponds to a handlebody-link.

\begin{figure}
\mbox{}\hfill
\begin{picture}(100,110)
%%% left inner circle %%%
\qbezier(45.5,63)(50,57)(50,50)
\qbezier(50,50)(48.5,31.5)(30,30)
\qbezier(30,30)(11.5,31.5)(10,50)
\qbezier(10,50)(11.5,68.5)(30,70)
\qbezier(30,70)(33,70)(35.5,69)
%%% left outer circle %%%
\qbezier(25,100)(25,80)(15.5,76.5)
\qbezier(15.5,76.5)(0.5,71)(0,50)
\qbezier(0,50)(0.5,29)(15.5,23.5)
\qbezier(15.5,23.5)(25,20)(25,0)
\qbezier(35,100)(35,80)(44.5,76.5)
\qbezier(54,68.5)(60,61)(60,50)
\qbezier(60,50)(59.5,29)(44.5,23.5)
\qbezier(44.5,23.5)(35,20)(35,0)
%%% right inner circle %%%
\qbezier(44.5,37)(40,43)(40,50)
\qbezier(40,50)(41.5,68.5)(60,70)
\qbezier(60,30)(57,30)(54.5,31)
\qbezier(60,30)(65,30)(70,30)
\qbezier(60,70)(65,70)(70,70)
%\qbezier(70,70)(88.5,68.5)(90,50)
\qbezier(80,67.5)(88.5,63.5)(90,50)
\qbezier(90,50)(88.5,31.5)(70,30)
%%% right outer circle %%%
\qbezier(36,31.5)(30,39)(30,50)
\qbezier(30,50)(32,78)(60,80)
\qbezier(60,20)(52,20)(45.5,23.5)
\qbezier(60,20)(65,20)(70,20)
\qbezier(60,80)(65,80)(70,80)
%\qbezier(70,80)(98,78)(100,50)
\qbezier(80,78)(98,73)(100,50)
\qbezier(100,50)(98,22)(70,20)
%%% right arc %%%
\qbezier(70,100)(70,65)(70,30)
\qbezier(80,100)(80,65)(80,33)
\qbezier(70,20)(70,10)(70,0)
\qbezier(80,21.5)(80,10)(80,0)
%%% thickness %%%
\qbezier(0,50)(5,46)(10,50)
\qbezier[6](0,50)(5,54)(10,50)
\qbezier(50,50)(55,46)(60,50)
\qbezier[6](50,50)(55,54)(60,50)
\qbezier(90,50)(95,46)(100,50)
\qbezier[6](90,50)(95,54)(100,50)
\qbezier(25,100)(30,96)(35,100)
\qbezier(25,100)(30,104)(35,100)
\qbezier(70,100)(75,96)(80,100)
\qbezier(70,100)(75,104)(80,100)
\qbezier(25,0)(30,-4)(35,0)
\qbezier[6](25,0)(30,4)(35,0)
\qbezier(70,0)(75,-4)(80,0)
\qbezier[6](70,0)(75,4)(80,0)
\end{picture}
\hfill\hfill
\begin{picture}(100,110)
\qbezier(0,0)(0,50)(0,100)
\qbezier(10,0)(10,50)(10,100)
\qbezier(0,100)(5,96)(10,100)
\qbezier(0,100)(5,104)(10,100)
\qbezier(0,0)(5,-4)(10,0)
\qbezier[6](0,0)(5,4)(10,0)
\qbezier(30,0)(30,50)(30,100)
\qbezier(40,0)(40,50)(40,100)
\qbezier(30,100)(35,96)(40,100)
\qbezier(30,100)(35,104)(40,100)
\qbezier(30,0)(35,-4)(40,0)
\qbezier[6](30,0)(35,4)(40,0)
\put(65,50){\makebox(0,0){$\cdots$}}
\qbezier(90,0)(90,50)(90,100)
\qbezier(100,0)(100,50)(100,100)
\qbezier(90,100)(95,96)(100,100)
\qbezier(90,100)(95,104)(100,100)
\qbezier(90,0)(95,-4)(100,0)
\qbezier[6](90,0)(95,4)(100,0)
\end{picture}
\hfill\mbox{}
\caption{}
\label{fig:handlebody-tangle}
\end{figure}

We define a strict tensor category $\mathcal{T}$ of handlebody-tangles
as follows.
The objects of $\mathcal{T}$ consist of finite sequences of disks.
We denote by the number $n$ the sequence of $n$ disks.
Then $\operatorname{Ob(\mathcal{T})}=\{0,1,2,\ldots\}$.
The morphisms of $\mathcal{T}$ are handlebody-tangles.
The source $s(T)$ and the target $b(T)$ of a handlebody-tangle $T$ with
$m$ top end disks and $n$ bottom end disks are defined by $s(T)=n$ and
$b(T)=m$.
The identity morphism $\operatorname{id}_n$ of an object $n$ is the
equivalence class of the trivial handlebody-tangle with $n$ top end
disks and $n$ bottom end disks, where the trivial handlebody-tangle is
the direct product of disks and the interval $I$ as shown in the right
picture of Figure~\ref{fig:handlebody-tangle}.
We remark that the identity morphism $\operatorname{id}_0$ is the empty set.
For handlebody-tangles $T,T'$ such that $s(T)=b(T')$, the composition
$T\circ T'$ of $T$ and $T'$ is the handlebody-tangle obtained by placing
$T$ on top of $T'$ and gluing the bottom end disks of $T$ and the top
end disks of $T'$ as shown in the left picture of
Figure~\ref{fig:composition_tensor}.
Then $\mathcal{T}$ is a category.
We equip $\mathcal{T}$ with a tensor product as follows.
For objects $m$, $n$ of $\mathcal{T}$, we define $m\otimes n:=m+n$.
For handlebody-tangles $T,T'$, the tensor product $T\otimes T'$ is the
handlebody-tangle obtained by placing $T'$ to the right of $T$ as show
in the right picture of Figure~\ref{fig:composition_tensor}.
Then the handlebody-tangle category $\mathcal{T}$ equipped with the
tensor product is a strict tensor category with the unit $0$.

\begin{figure}
\mbox{}\hfill
\begin{picture}(90,140)
\qbezier(10,130)(15,126)(20,130)
\qbezier(10,130)(15,134)(20,130)
\qbezier(10,30)(10,40)(10,50)
\qbezier(20,30)(20,40)(20,50)
\qbezier(10,70)(10,80)(10,90)
\qbezier(20,70)(20,80)(20,90)
\qbezier(10,110)(10,120)(10,130)
\qbezier(20,110)(20,120)(20,130)
\qbezier(10,30)(15,26)(20,30)
\qbezier[6](10,30)(15,34)(20,30)
\qbezier(30,130)(35,126)(40,130)
\qbezier(30,130)(35,134)(40,130)
\qbezier(30,30)(30,40)(30,50)
\qbezier(40,30)(40,40)(40,50)
\qbezier(30,70)(30,80)(30,90)
\qbezier(40,70)(40,80)(40,90)
\qbezier(30,110)(30,120)(30,130)
\qbezier(40,110)(40,120)(40,130)
\qbezier(30,30)(35,26)(40,30)
\qbezier[6](30,30)(35,34)(40,30)
\put(55,120){\makebox(0,0){$\cdots$}}
\put(55,80){\makebox(0,0){$\cdots$}}
\put(55,40){\makebox(0,0){$\cdots$}}
\qbezier(70,130)(75,126)(80,130)
\qbezier(70,130)(75,134)(80,130)
\qbezier(70,30)(70,40)(70,50)
\qbezier(80,30)(80,40)(80,50)
\qbezier(70,70)(70,80)(70,90)
\qbezier(80,70)(80,80)(80,90)
\qbezier(70,110)(70,120)(70,130)
\qbezier(80,110)(80,120)(80,130)
\qbezier(70,30)(75,26)(80,30)
\qbezier[6](70,30)(75,34)(80,30)
\qbezier(0,110)(45,110)(90,110)
\qbezier(0,90)(45,90)(90,90)
\qbezier(0,90)(0,100)(0,110)
\qbezier(90,90)(90,100)(90,110)
\put(45,100){\makebox(0,0){$T$}}
\qbezier(0,70)(45,70)(90,70)
\qbezier(0,50)(45,50)(90,50)
\qbezier(0,50)(0,60)(0,70)
\qbezier(90,50)(90,60)(90,70)
\put(45,60){\makebox(0,0){$T'$}}
\put(45,10){\makebox(0,0){$T\circ T'$}}
\end{picture}
\hfill\hfill\hfill
\begin{picture}(190,140)
\qbezier(10,130)(15,126)(20,130)
\qbezier(10,130)(15,134)(20,130)
\qbezier(10,30)(10,45)(10,60)
\qbezier(20,30)(20,45)(20,60)
\qbezier(10,100)(10,115)(10,130)
\qbezier(20,100)(20,115)(20,130)
\qbezier(10,30)(15,26)(20,30)
\qbezier[6](10,30)(15,34)(20,30)
\qbezier(30,130)(35,126)(40,130)
\qbezier(30,130)(35,134)(40,130)
\qbezier(30,30)(30,45)(30,60)
\qbezier(40,30)(40,45)(40,60)
\qbezier(30,100)(30,115)(30,130)
\qbezier(40,100)(40,115)(40,130)
\qbezier(30,30)(35,26)(40,30)
\qbezier[6](30,30)(35,34)(40,30)
\put(55,115){\makebox(0,0){$\cdots$}}
\put(55,45){\makebox(0,0){$\cdots$}}
\qbezier(70,130)(75,126)(80,130)
\qbezier(70,130)(75,134)(80,130)
\qbezier(70,30)(70,45)(70,60)
\qbezier(80,30)(80,45)(80,60)
\qbezier(70,100)(70,115)(70,130)
\qbezier(80,100)(80,115)(80,130)
\qbezier(70,30)(75,26)(80,30)
\qbezier[6](70,30)(75,34)(80,30)
\qbezier(0,100)(45,100)(90,100)
\qbezier(0,60)(45,60)(90,60)
\qbezier(0,60)(0,80)(0,100)
\qbezier(90,60)(90,80)(90,100)
\put(45,80){\makebox(0,0){$T$}}
\qbezier(110,130)(115,126)(120,130)
\qbezier(110,130)(115,134)(120,130)
\qbezier(110,30)(110,45)(110,60)
\qbezier(120,30)(120,45)(120,60)
\qbezier(110,100)(110,115)(110,130)
\qbezier(120,100)(120,115)(120,130)
\qbezier(110,30)(115,26)(120,30)
\qbezier[6](110,30)(115,34)(120,30)
\qbezier(130,130)(135,126)(140,130)
\qbezier(130,130)(135,134)(140,130)
\qbezier(130,30)(130,45)(130,60)
\qbezier(140,30)(140,45)(140,60)
\qbezier(130,100)(130,115)(130,130)
\qbezier(140,100)(140,115)(140,130)
\qbezier(130,30)(135,26)(140,30)
\qbezier[6](130,30)(135,34)(140,30)
\put(155,115){\makebox(0,0){$\cdots$}}
\put(155,45){\makebox(0,0){$\cdots$}}
\qbezier(170,130)(175,126)(180,130)
\qbezier(170,130)(175,134)(180,130)
\qbezier(170,30)(170,45)(170,60)
\qbezier(180,30)(180,45)(180,60)
\qbezier(170,100)(170,115)(170,130)
\qbezier(180,100)(180,115)(180,130)
\qbezier(170,30)(175,26)(180,30)
\qbezier[6](170,30)(175,34)(180,30)
\qbezier(100,100)(145,100)(190,100)
\qbezier(100,60)(145,60)(190,60)
\qbezier(100,60)(100,80)(100,100)
\qbezier(190,60)(190,80)(190,100)
\put(145,80){\makebox(0,0){$T'$}}
\put(95,10){\makebox(0,0){$T\otimes T'$}}
\end{picture}
\hfill\mbox{}
\caption{}
\label{fig:composition_tensor}
\end{figure}

We give generators and relations for the strict tensor category $\mathcal{T}$.
Every morphism in $\mathcal{T}$ can be presented by the generators
with the operations of composing and tensoring applied.
Two morphisms are identical if and only if they, presented as above,
deform to each other by using the relations.
We refer the reader to~\cite[Chapter XII]{Kassel95} for details of
generators and relations for a strict tensor category.
Let $\,|\,$, $\cap$, $\cup$, $\curlywedge$, $\curlyvee$, $X$, and
$\overline{X}$ be the handlebody-tangles depicted in
Figure~\ref{fig:elementary_handlebody-tangle}.
The following proposition immediately follows from Theorem~11
in~\cite{IshiharaIshii12}.

\begin{figure}
\mbox{}\hfill
\begin{picture}(10,90)
\qbezier(0,80)(5,76)(10,80)
\qbezier(0,80)(5,84)(10,80)
\qbezier(0,30)(0,55)(0,80)
\qbezier(10,30)(10,55)(10,80)
\qbezier(0,30)(5,26)(10,30)
\qbezier[6](0,30)(5,34)(10,30)
\put(5,10){\makebox(0,0){$|$}}
\end{picture}
\hfill
\begin{picture}(40,90)
\qbezier(0,30)(0,60)(20,60)
\qbezier(40,30)(40,60)(20,60)
\qbezier(10,30)(10,50)(20,50)
\qbezier(30,30)(30,50)(20,50)
\qbezier(0,30)(5,26)(10,30)
\qbezier[6](0,30)(5,34)(10,30)
\qbezier(30,30)(35,26)(40,30)
\qbezier[6](30,30)(35,34)(40,30)
\put(20,10){\makebox(0,0){$\cap$}}
\end{picture}
\hfill
\begin{picture}(40,90)
\qbezier(0,80)(5,76)(10,80)
\qbezier(0,80)(5,84)(10,80)
\qbezier(30,80)(35,76)(40,80)
\qbezier(30,80)(35,84)(40,80)
\qbezier(0,80)(0,50)(20,50)
\qbezier(40,80)(40,50)(20,50)
\qbezier(10,80)(10,60)(20,60)
\qbezier(30,80)(30,60)(20,60)
\put(20,10){\makebox(0,0){$\cup$}}
\end{picture}
\hfill
\begin{picture}(40,90)
\qbezier(15,80)(20,76)(25,80)
\qbezier(15,80)(20,84)(25,80)
\qbezier(10,58)(15,60.5)(15,80)
\qbezier(30,58)(25,60.5)(25,80)
\qbezier(0,30)(0,53)(10,58)
\qbezier(40,30)(40,53)(30,58)
\qbezier(10,30)(10,50)(20,50)
\qbezier(30,30)(30,50)(20,50)
\qbezier(0,30)(5,26)(10,30)
\qbezier[6](0,30)(5,34)(10,30)
\qbezier(30,30)(35,26)(40,30)
\qbezier[6](30,30)(35,34)(40,30)
\put(20,10){\makebox(0,0){$\curlywedge$}}
\end{picture}
\hfill
\begin{picture}(40,90)
\qbezier(0,80)(5,76)(10,80)
\qbezier(0,80)(5,84)(10,80)
\qbezier(30,80)(35,76)(40,80)
\qbezier(30,80)(35,84)(40,80)
\qbezier(10,52)(15,49.5)(15,30)
\qbezier(30,52)(25,49.5)(25,30)
\qbezier(0,80)(0,57)(10,52)
\qbezier(40,80)(40,57)(30,52)
\qbezier(10,80)(10,60)(20,60)
\qbezier(30,80)(30,60)(20,60)
\qbezier(15,30)(20,26)(25,30)
\qbezier[6](15,30)(20,34)(25,30)
\put(20,10){\makebox(0,0){$\curlyvee$}}
\end{picture}
\hfill
\begin{picture}(40,90)
\qbezier(0,80)(5,76)(10,80)
\qbezier(0,80)(5,84)(10,80)
\qbezier(30,80)(35,76)(40,80)
\qbezier(30,80)(35,84)(40,80)
\qbezier(0,30)(0,42)(20,62)
\qbezier(20,62)(30,72)(30,80)
\qbezier(10,30)(10,38)(20,48)
\qbezier(20,48)(40,68)(40,80)
\qbezier(40,30)(40,42)(27,55)
\qbezier(20,62)(10,72)(10,80)
\qbezier(30,30)(30,38)(20,48)
\qbezier(13,55)(0,68)(0,80)
\qbezier(0,30)(5,26)(10,30)
\qbezier[6](0,30)(5,34)(10,30)
\qbezier(30,30)(35,26)(40,30)
\qbezier[6](30,30)(35,34)(40,30)
\put(20,10){\makebox(0,0){$X$}}
\end{picture}
\hfill
\begin{picture}(40,90)
\qbezier(40,80)(35,76)(30,80)
\qbezier(40,80)(35,84)(30,80)
\qbezier(10,80)(5,76)(0,80)
\qbezier(10,80)(5,84)(0,80)
\qbezier(40,30)(40,42)(20,62)
\qbezier(20,62)(10,72)(10,80)
\qbezier(30,30)(30,38)(20,48)
\qbezier(20,48)(0,68)(0,80)
\qbezier(0,30)(0,42)(13,55)
\qbezier(20,62)(30,72)(30,80)
\qbezier(10,30)(10,38)(20,48)
\qbezier(27,55)(40,68)(40,80)
\qbezier(40,30)(35,26)(30,30)
\qbezier[6](40,30)(35,34)(30,30)
\qbezier(10,30)(5,26)(0,30)
\qbezier[6](10,30)(5,34)(0,30)
\put(20,10){\makebox(0,0){$\overline{X}$}}
\end{picture}
\hfill\mbox{}
\caption{}
\label{fig:elementary_handlebody-tangle}
\end{figure}

\begin{proposition} \label{prop:old-relations}
The strict tensor category $\mathcal{T}$ is generated by the six morphisms
\[ \cap,\,\cup,\,\curlywedge,\,\curlyvee,\,X,\,\overline{X} \]
and the relations
\begin{align}
& (\mmid\mcap)\circ(\mcup\mmid)=\mmid=(\mcap\mmid)\circ(\mmid\mcup), \label{eq:old1} \\
& (\mmid\mcap)\circ(X\mmid)=(\mcap\mmid)\circ(\mmid\overline{X}), \label{eq:old2} \\
& (\mmid\mcap)\circ(\overline{X}\mmid)=(\mcap\mmid)\circ(\mmid X), \label{eq:old3} \\
& (\mmid\mcap)\circ(\mvee\mmid)=\mwedge=(\mcap\mmid)\circ(\mmid\mvee), \label{eq:old4} \\
& (\mmid\mcap)\circ(X\mmid)\circ(\mmid\mcup)=\mmid
=(\mmid\mcap)\circ(\overline{X}\mmid)\circ(\mmid\mcup), \label{eq:old5} \\
& X\circ\overline{X}=\mmid\mmid, \label{eq:old6} \\
& (X\mmid)\circ(\mmid X)\circ(X\mmid)=(\mmid X)\circ(X\mmid)\circ(\mmid X), \label{eq:old7} \\
& \mwedge\circ X=\mwedge=\mwedge\circ\overline{X}, \label{eq:old8} \\
& (\mwedge\mmid)\circ(\mmid X)\circ(X\mmid)=X\circ(\mmid\mwedge), \label{eq:old9} \\
& (\mmid\mwedge)\circ(X\mmid)\circ(\mmid X)=X\circ(\mwedge\mmid), \label{eq:old10} \\
& \mwedge\circ(\mwedge\mmid)=\mwedge\circ(\mmid\mwedge), \label{eq:old11}
\end{align}
where we have presented $A\otimes B$ by the juxtaposition $AB$.
\end{proposition}

We improve below the presentation of $\mathcal{T}$ above into a more
economical one, which will be crucial when we prove Proposition~\ref{prop:F_A1}.

\begin{proposition} \label{prop:new-relations}
The strict tensor category $\mathcal{T}$ is generated by the five morphisms
\[ \cap,\,\cup,\,\curlywedge,\,X,\,\overline{X} \]
and the relations \eqref{eq:old1}, \eqref{eq:old6}, \eqref{eq:old7},
\eqref{eq:old9}, \eqref{eq:old10}, \eqref{eq:old11} together with
\begin{align}
& \mcap\circ(\mwedge\mmid)=\mcap\circ(\mmid\mwedge), \label{eq:new1} \\
& \mcap\circ X=\mcap, \label{eq:new2} \\
& (\mcap\mmid)\circ(\mmid X)\circ(X\mmid)=\mmid\mcap, \label{eq:new3} \\
& (\mmid\mcap)\circ(X\mmid)\circ(\mmid X)=\mcap\mmid, \label{eq:new4} \\
& \mwedge\circ X=\mwedge. \label{eq:new5}
\end{align}
\end{proposition}

\begin{proof}
The five morphisms generate $\mathcal{T}$ since we have
\begin{align*}
\mvee=(\mwedge\mmid)\circ(\mmid\mcup).
\end{align*}
By this equality the relation \eqref{eq:old4} turns into the relation
\begin{align}
(\mwedge\,\mcap)\circ(\mmid\mcup\mmid)=\mwedge
=(\mcap\mmid)\circ(\mmid\mwedge\mmid)\circ(\mmid\mmid\mcup)
\label{eq:old4'}
\end{align}
We see that the conditions in Proposition~\ref{prop:new-relations}
follow from those in Proposition~\ref{prop:old-relations}.
To prove the converse it suffices to verify \eqref{eq:old2},
\eqref{eq:old3}, \eqref{eq:old5} and \eqref{eq:old4'}.
This is done as follows.
\begin{align*}
& (\mmid\mcap)\circ(X\mmid)
=(\mmid\mcap)\circ(X\mmid)\circ(\mmid X)\circ(\mmid\overline{X})
=(\mcap\mmid)\circ(\mmid\overline{X}), \\
& (\mmid\mcap)\circ(\overline{X}\mmid)
=(\mcap\mmid)\circ(\mmid X)\circ(X\mmid)\circ(\overline{X}\mmid)
=(\mcap\mmid)\circ(\mmid X), \\
& (\mmid\mcap)\circ(X\mmid)\circ(\mmid\mcup)
=(\mmid\mcap\,\mcap)\circ(\mcup X\mmid)\circ(\mmid\mcup) \\
&\hspace{5mm}=(\mmid\mcap\,\mcap)\circ(\mmid\mmid X\mmid)\circ(\mmid X\mcup)
\circ(\mmid\overline{X})\circ(\mcup\mmid) \\
&\hspace{5mm}=(\mmid\mcap)\circ(\mmid\mmid\mcap\mmid)
\circ(\mmid\overline{X}\mmid\mmid)\circ(\mcup\mmid\mcup)
=(\mmid\mcap)\circ(\mmid\overline{X})\circ(\mcup\mmid) \\
&\hspace{5mm}=(\mmid\mcap)\circ(\mmid X)\circ(\mmid\overline{X})\circ(\mcup\mmid)
=(\mmid\mcap)\circ(\mcup\mmid)=\mmid, \\
& (\mmid\mcap)\circ(\overline{X}\mmid)\circ(\mmid\mcup)
=(\mmid\mcap\,\mcap)\circ(\mcup\overline{X}\mmid)\circ(\mmid\mcup) \\
&\hspace{5mm}=(\mmid\mcap)\circ(\mmid\mmid\mcap\mmid)\circ(\mmid X\mmid\mmid)
\circ(\mcup X\mmid)\circ(\overline{X}\mmid)\circ(\mmid\mcup) \\
&\hspace{5mm}=(\mmid\mcap)\circ(\mmid\mmid\mcap\mmid)\circ(\mmid X\mcup)\circ(\mcup\mmid)
=(\mmid\mcap)\circ(\mmid X)\circ(\mcup\mmid) \\
&\hspace{5mm}=(\mmid\mcap)\circ(\mcup\mmid)=\mmid, \\
& \mwedge=\mwedge\circ X\circ\overline{X}=\mwedge\circ\overline{X}, \\
& (\mwedge\,\mcap)\circ(\mmid\mcup\mmid)
=\mwedge
=(\mcap\mmid)\circ(\mwedge\,\mcup)
=(\mcap\mmid)\circ(\mmid\mwedge\mmid)\circ(\mmid\mmid\mcup).
\end{align*}
\end{proof}

The following is now easy to see.

\begin{proposition} \label{prop:Tbraided}
The tensor category $\mathcal{T}$ is braided with respect to the
braiding depicted in Figure~\ref{fig:braiding}.
\end{proposition}

\begin{figure}
\mbox{}\hfill
\begin{picture}(150,140)
\qbezier(10,130)(15,126)(20,130)
\qbezier(10,130)(15,134)(20,130)
\qbezier(50,130)(55,126)(60,130)
\qbezier(50,130)(55,134)(60,130)
\put(42.5,115){\makebox(0,0){$\cdots$}}
\qbezier(80,130)(85,126)(90,130)
\qbezier(80,130)(85,134)(90,130)
\qbezier(105,130)(110,126)(115,130)
\qbezier(105,130)(110,134)(115,130)
\put(121,115){\makebox(0,0){$\cdots$}}
\qbezier(140,130)(145,126)(150,130)
\qbezier(140,130)(145,134)(150,130)
\qbezier(0,30)(0,42)(39,80)
\qbezier(10,30)(10,40)(51,80)
\qbezier(39,80)(80,120)(80,130)
\qbezier(51,80)(90,118)(90,130)
\qbezier(25,30)(25,42)(64,80)
\qbezier(35,30)(35,40)(76,80)
\qbezier(64,80)(105,120)(105,130)
\qbezier(76,80)(115,118)(115,130)
\qbezier(60,30)(60,42)(99,80)
\qbezier(70,30)(70,40)(111,80)
\qbezier(99,80)(140,120)(140,130)
\qbezier(111,80)(150,118)(150,130)
%
%\qbezier(100,30)(100,42)(61,80)
%\qbezier(90,30)(90,40)(49,80)
%\qbezier(61,80)(20,120)(20,130)
%\qbezier(49,80)(10,118)(10,130)
%\qbezier(140,30)(140,42)(101,80)
%\qbezier(130,30)(130,40)(89,80)
%\qbezier(101,80)(60,120)(60,130)
%\qbezier(89,80)(50,118)(50,130)
%
\qbezier(100,30)(100,37)(86,54)
\qbezier(90,30)(90,35)(80,47)
\qbezier(80,60.5)(74.5,66.5)(68.5,72.5)
\qbezier(74,54)(68.5,60.5)(62.5,66.5)
\qbezier(62.5,78.5)(59,82)(56,85)
\qbezier(56.5,72.5)(53,76)(50,79)
\qbezier(50,91)(20,121)(20,130)
\qbezier(44,85)(10,119)(10,130)
\qbezier(140,30)(140,41)(106,75)
\qbezier(130,30)(130,39)(100,69)
\qbezier(100,81)(94,87)(88.5,92.5)
\qbezier(94,75)(88,81)(82.5,86.5)
\qbezier(82.5,99)(79,102.5)(76,106)
\qbezier(76.5,92.5)(73,96)(70,99.5)
\qbezier(70,113)(60,125)(60,130)
\qbezier(64,106)(50,123)(50,130)
\qbezier(0,30)(5,26)(10,30)
\qbezier[6](0,30)(5,34)(10,30)
\qbezier(25,30)(30,26)(35,30)
\qbezier[6](25,30)(30,34)(35,30)
\put(56,45){\makebox(0,0){$\cdots$}}
\qbezier(60,30)(65,26)(70,30)
\qbezier[6](60,30)(65,34)(70,30)
\qbezier(90,30)(95,26)(100,30)
\qbezier[6](90,30)(95,34)(100,30)
\put(107.5,45){\makebox(0,0){$\cdots$}}
\qbezier(130,30)(135,26)(140,30)
\qbezier[6](130,30)(135,34)(140,30)
\put(35,10){\makebox(0,0){$\underbrace{\hspace{70pt}}_m$}}
\put(115,10){\makebox(0,0){$\underbrace{\hspace{50pt}}_n$}}
\end{picture}
\hfill\mbox{}
\caption{}
\label{fig:braiding}
\end{figure}

\section{Quantum-commutative quantum-symmetric algebras}
\label{sec:q-symmetric-alg}

In this section, $\mathcal{B}=(\mathcal{B},\otimes,I)$
denotes a braided tensor category unless otherwise stated.
The braiding and its inverse will be denoted by, and depicted as
\begin{equation*}
c_{V,W}:V\otimes W\overset{\simeq}{\longrightarrow}W\otimes V;
\begin{minipage}{40pt}
\begin{picture}(40,50)
 \qbezier(10,15)(20,25)(30,35)
 \qbezier(10,35)(14,31)(18,27)
 \qbezier(30,15)(26,19)(22,23)
 \put(10,43){\makebox(0,0){\small$W$}}
 \put(30,43){\makebox(0,0){\small$V$}}
 \put(10,7){\makebox(0,0){\small$V$}}
 \put(30,7){\makebox(0,0){\small$W$}}
\end{picture}
\end{minipage}
\end{equation*}
\begin{equation*}
c_{W,V}^{-1}:W\otimes V\overset{\simeq}{\longrightarrow}V\otimes W;
\begin{minipage}{40pt}
\begin{picture}(40,50)
 \qbezier(30,15)(20,25)(10,35)
 \qbezier(30,35)(26,31)(22,27)
 \qbezier(10,15)(14,19)(18,23)
 \put(10,43){\makebox(0,0){\small$V$}}
 \put(30,43){\makebox(0,0){\small$W$}}
 \put(10,7){\makebox(0,0){\small$W$}}
 \put(30,7){\makebox(0,0){\small$V$}}
\end{picture}
\end{minipage}
\end{equation*}
where $V,W\in\mathrm{Ob}(\mathcal{B})$.
Suppose that $A$ is a (non-unital) algebra in $\mathcal{B}$.
Thus, $A$ is equipped with a morphism in $\mathcal{B}$
\begin{equation*}
m_A:A\otimes A\to A;
\begin{minipage}{40pt}
\begin{picture}(40,50)
 \qbezier(10,15)(15,20)(20,25)
 \qbezier(30,15)(25,20)(20,25)
 \qbezier(20,25)(20,30)(20,35)
 \put(20,43){\makebox(0,0){\small$A$}}
 \put(10,7){\makebox(0,0){\small$A$}}
 \put(30,7){\makebox(0,0){\small$A$}}
\end{picture}
\end{minipage}
\end{equation*}
which satisfies the associativity
\[ \begin{minipage}{60pt}
\begin{picture}(60,60)
 \qbezier(10,15)(20,25)(30,35)
 \qbezier(50,15)(40,25)(30,35)
 \qbezier(30,15)(25,20)(20,25)
 \qbezier(30,35)(30,40)(30,45)
 \put(30,53){\makebox(0,0){\small$A$}}
 \put(10,7){\makebox(0,0){\small$A$}}
 \put(30,7){\makebox(0,0){\small$A$}}
 \put(50,7){\makebox(0,0){\small$A$}}
\end{picture}
\end{minipage}
=\begin{minipage}{60pt}
\begin{picture}(60,60)
 \qbezier(10,15)(20,25)(30,35)
 \qbezier(50,15)(40,25)(30,35)
 \qbezier(30,15)(35,20)(40,25)
 \qbezier(30,35)(30,40)(30,45)
 \put(30,53){\makebox(0,0){\small$A$}}
 \put(10,7){\makebox(0,0){\small$A$}}
 \put(30,7){\makebox(0,0){\small$A$}}
 \put(50,7){\makebox(0,0){\small$A$}}
\end{picture}
\end{minipage}~. \]
Here and in what follows we omit the associativity constraint in
diagrams.
\begin{definition}\label{def:q-symmetric-alg}
A (non-unital) algebra $A$ in $\mathcal{B}$, equipped with two morphisms
\begin{equation*}
\mathrm{ev}_A:A\otimes A\to I;
\begin{minipage}{40pt}
\begin{picture}(40,30)
 \qbezier(10,15)(10,25)(20,25)
 \qbezier(30,15)(30,25)(20,25)
 \put(10,7){\makebox(0,0){\small$A$}}
 \put(30,7){\makebox(0,0){\small$A$}}
\end{picture}
\end{minipage}
\end{equation*}
\begin{equation*}
\mathrm{coev}_A:I\to A\otimes A;
\begin{minipage}{40pt}
\begin{picture}(40,30)
 \qbezier(10,15)(10,5)(20,5)
 \qbezier(30,15)(30,5)(20,5)
 \put(10,23){\makebox(0,0){\small$A$}}
 \put(30,23){\makebox(0,0){\small$A$}}
\end{picture}
\end{minipage}
\end{equation*}
is called a \emph{quantum-commutative quantum-symmetric algebra}
if it satisfies the selfduality~\eqref{eq:selfduality},
the Frobenius property~\eqref{eq:Frobenius},
the quantum-commutativity~\eqref{eq:q-commutativity}
and the quantum-symmetry~\eqref{eq:q-symmetry}, given below.
\begin{align}
&\begin{minipage}{60pt}
\begin{picture}(60,70)
 \qbezier(10,35)(10,45)(20,45)
 \qbezier(30,35)(30,45)(20,45)
 \qbezier(50,35)(50,45)(50,55)
 \qbezier(10,15)(10,25)(10,35)
 \qbezier(30,35)(30,25)(40,25)
 \qbezier(50,35)(50,25)(40,25)
 \put(50,63){\makebox(0,0){\small$A$}}
 \put(10,7){\makebox(0,0){\small$A$}}
\end{picture}
\end{minipage}
=\begin{minipage}{60pt}
\begin{picture}(60,70)
 \qbezier(10,35)(10,45)(10,55)
 \qbezier(30,35)(30,45)(40,45)
 \qbezier(50,35)(50,45)(40,45)
 \qbezier(10,35)(10,25)(20,25)
 \qbezier(30,35)(30,25)(20,25)
 \qbezier(50,15)(50,25)(50,35)
 \put(10,63){\makebox(0,0){\small$A$}}
 \put(50,7){\makebox(0,0){\small$A$}}
\end{picture}
\end{minipage}
=\begin{minipage}{20pt}
\begin{picture}(20,70)
 \qbezier(10,15)(10,35)(10,55)
 \put(10,63){\makebox(0,0){\small$A$}}
 \put(10,7){\makebox(0,0){\small$A$}}
\end{picture}
\end{minipage}~\text{, the identity map}
\label{eq:selfduality} \\
&\begin{minipage}{60pt}
\begin{picture}(60,50)
 \qbezier(10,15)(15,20)(20,25)
 \qbezier(30,15)(25,20)(20,25)
 \qbezier(50,15)(50,20)(50,25)
 \qbezier(20,25)(20,40)(35,40)
 \qbezier(50,25)(50,40)(35,40)
 \put(10,7){\makebox(0,0){\small$A$}}
 \put(30,7){\makebox(0,0){\small$A$}}
 \put(50,7){\makebox(0,0){\small$A$}}
\end{picture}
\end{minipage}
=\begin{minipage}{60pt}
\begin{picture}(60,50)
 \qbezier(10,15)(10,20)(10,25)
 \qbezier(30,15)(35,20)(40,25)
 \qbezier(50,15)(45,20)(40,25)
 \qbezier(10,25)(10,40)(25,40)
 \qbezier(40,25)(40,40)(25,40)
 \put(10,7){\makebox(0,0){\small$A$}}
 \put(30,7){\makebox(0,0){\small$A$}}
 \put(50,7){\makebox(0,0){\small$A$}}
\end{picture}
\end{minipage}
\label{eq:Frobenius} \\
&\begin{minipage}{40pt}
\begin{picture}(40,70)
 \qbezier(10,35)(15,40)(20,45)
 \qbezier(30,35)(25,40)(20,45)
 \qbezier(20,45)(20,50)(20,55)
 \qbezier(10,15)(20,25)(30,35)
 \qbezier(10,35)(14,31)(18,27)
 \qbezier(30,15)(26,19)(22,23)
 \put(20,63){\makebox(0,0){\small$A$}}
 \put(10,7){\makebox(0,0){\small$A$}}
 \put(30,7){\makebox(0,0){\small$A$}}
\end{picture}
\end{minipage}
=\begin{minipage}{40pt}
\begin{picture}(40,70)
 \qbezier(20,40)(20,47.5)(20,55)
 \qbezier(10,30)(15,35)(20,40)
 \qbezier(30,30)(25,35)(20,40)
 \qbezier(10,15)(10,22.5)(10,30)
 \qbezier(30,15)(30,22.5)(30,30)
 \put(20,63){\makebox(0,0){\small$A$}}
 \put(10,7){\makebox(0,0){\small$A$}}
 \put(30,7){\makebox(0,0){\small$A$}}
\end{picture}
\end{minipage}
\label{eq:q-commutativity} \\
&\begin{minipage}{40pt}
\begin{picture}(40,55)
 \qbezier(10,35)(10,45)(20,45)
 \qbezier(30,35)(30,45)(20,45)
 \qbezier(10,15)(20,25)(30,35)
 \qbezier(10,35)(14,31)(18,27)
 \qbezier(30,15)(26,19)(22,23)
 \put(10,7){\makebox(0,0){\small$A$}}
 \put(30,7){\makebox(0,0){\small$A$}}
\end{picture}
\end{minipage}
=\begin{minipage}{40pt}
\begin{picture}(40,55)
 \qbezier(10,30)(10,40)(20,40)
 \qbezier(30,30)(30,40)(20,40)
 \qbezier(10,15)(10,22.5)(10,30)
 \qbezier(30,15)(30,22.5)(30,30)
 \put(10,7){\makebox(0,0){\small$A$}}
 \put(30,7){\makebox(0,0){\small$A$}}
\end{picture}
\end{minipage}
\label{eq:q-symmetry}
\end{align}
These coincide respectively with \eqref{eq:old1}, \eqref{eq:new1},
\eqref{eq:new5}, \eqref{eq:new2} given before,
if the $\curlywedge,\ \cap,\ \cup$ and $X$ before read
$m_A,\mathrm{ev}_A,\mathrm{coev}_A$ and $c_{A,A}$, respectively.
\end{definition}

\begin{proposition}\label{prop:q-symmetric-alg}
\begin{itemize}
\item[(1)]
In the braided tensor category $\mathcal{T}$ of handlebody-tangles, the
object $1$ equipped with $\curlywedge,\ \cap,\ \cup$ is a
quantum-commutative quantum-symmetric algebra.
\item[(2)]
If $F:\mathcal{T}\to\mathcal{B}$ is a braided tensor functor, then the
object $F(1)$, together with the morphisms which are given by
$F(\curlywedge)$, $F(\cap)$, $F(\cup)$ composed with the tensor
structure $I\simeq F(0)$, $F(1)\otimes F(1)\simeq F(2)$ of $F$, forms a
quantum-commutative quantum-symmetric algebra in $\mathcal{B}$.
\end{itemize}
\end{proposition}

\begin{proof}
\begin{itemize}
\item[(1)]
This follows by Proposition~\ref{prop:new-relations}.
\item[(2)]
This follows by a standard argument which depends on the fact that the
notion of quantum-commutative quantum-symmetric algebras is
tensor-categorical.
\end{itemize}
\end{proof}

\begin{proposition}\label{prop:F_A1}
Assume that $\mathcal{B}$ is strict as a tensor category.
Given a quantum-commutative quantum-symmetric algebra
$A=(A,m_A,\mathrm{ev}_A,\mathrm{coev}_A)$ in $\mathcal{B}$,
there uniquely exists a braided strict tensor functor
$F_A:\mathcal{T}\to\mathcal{B}$ such that
\begin{equation} \label{eq:braided-functor}
F_A(1) = A, ~
F_A(\curlywedge)=m_A, ~
F_A(\cap)=\mathrm{ev}_A, ~
F_A(\cup) = \mathrm{coev}_A.
\end{equation}
\end{proposition}
\begin{proof}
Define objects $A^{\otimes n} \in \mathrm{Ob}(\mathcal{B})$ by
\begin{equation}\label{eq:tensor-power}
A^{\otimes 0}=I, ~
A^{\otimes 1}=A, ~
A^{\otimes n}=A^{\otimes(n-1)}\otimes A ~(n>1),
\end{equation}
and define a map $F:\mathrm{Ob}(\mathcal{T})\to\mathrm{Ob}(\mathcal{B})$
by
\begin{equation}\label{eq:FonOb}
F(n)=A^{\otimes n}.
\end{equation}
Then this $F$ strictly preserves the tensor product.

Let us see that the relations in Proposition~\ref{prop:new-relations}
with $X$, $\overline{X}$, $\curlywedge$, $\cap$, $\cup$ replaced by
$c_{A,A}$, $c_{A,A}^{-1}$, $m_A$, $\mathrm{ev}_A$, $\mathrm{coev}_A$ are satisfied.
Since $c_{A,A}$ is a braiding, the relations \eqref{eq:old7},
\eqref{eq:old9}, \eqref{eq:old10}, \eqref{eq:new3} and \eqref{eq:new4}
are satisfied; cf. Remark~\ref{rem:braided-alg} below.
The rest is satisfied by the coincidence noted below
\eqref{eq:selfduality}--\eqref{eq:q-symmetry} and
since $c_{A,A}$ and $c_{A,A}^{-1}$ are inverses of each other.

It follows by \cite[Proposition XII.1.4]{Kassel95} that the map $F$
defined above gives rise uniquely to a strict tensor functor
$F_A:\mathcal{T}\to\mathcal{B}$ which satisfies the equalities in
\eqref{eq:braided-functor} as well as $F_A(X)=c_{A,A}$,
$F_A(\overline{X})=c_{A,A}^{-1}$.
One sees easily that this $F_A$ preserves the braiding, and is indeed a
unique functor such as described above.
\end{proof} 

\begin{remark} \label{rem:braided-alg}
Assume that $\mathcal{B}$ is strict, but is not necessarily braided.
As was essentially shown above, if we have a (non-unital) algebra $A$ in $\mathcal{B}$
equipped with morphisms $c_{A,A},c_{A,A}^{-1},\mathrm{ev}_A,\mathrm{coev}_A$
which satisfy the relations in Proposition~\ref{prop:new-relations} except \eqref{eq:old11},
there uniquely exists a strict tensor functor $F_A:\mathcal{T}\to\mathcal{B}$
which satisfies the equalities in \eqref{eq:braided-functor}
as well as $F_A(X)=c_{A,A}$, $F_A(\overline{X})=c_{A,A}^{-1}$.
In fact, it is essentially from such functors that Ishihara and the
first author~\cite{IshiharaIshii12} constructed invariants of
handlebody-knots.
However, we choose, in this paper, to work with braided tensor categories,
because the algebras of our main concern sit in such a category.
\end{remark}

Suppose that $\mathcal{B}$ is braided.
To modify the proposition above when $\mathcal{B}$ is not necessarily
strict, let us say that a tensor functor $F:\mathcal{T}\to\mathcal{B}$
is \emph{almost strict}, if it satisfies the following three conditions:
(i) $F(0)=I$,
(ii) if we set $A=F(1)$, then for each $n>1$, $F(n)=A^{\otimes n}$,
where $A^{\otimes n}$ is defined by \eqref{eq:tensor-power}, and
(iii) the tensor structure $\varphi_0$, $\varphi_2$ of $F$ is as
follows,
\begin{itemize}
\item[(a)]
$\varphi_0:I\overset{\simeq}{\longrightarrow}F(0)$ is the identity $\mathrm{id}_I$;
\item[(b)]
for every $n,m>0$,
$\varphi_2(n,m):F(n)\otimes F(m)\overset{\simeq}{\longrightarrow}F(n+m)$
is \emph{canonical} in the sense that it is built from the associativity
constraint
$a_{A,A,A}:(A\otimes A)\otimes A\overset{\simeq}{\longrightarrow}A\otimes(A\otimes A)$
and the identity $\mathrm{id}_A$ by composing and tensoring; such an
isomorphism is unique by MacLane's coherence theorem.
\end{itemize}
One sees from (a) that if $n=0$ or $m=0$, then $\varphi_2(n,m)$ must
coincide with the left or right unit constraint
$l_{A^{\otimes m}}:I\otimes A^{\otimes m}\overset{\simeq}{\longrightarrow}A^{\otimes m}$,
$r_{A^{\otimes n}}:A^{\otimes n}\otimes I\overset{\simeq}{\longrightarrow}A^{\otimes n}$.

\begin{proposition}\label{prop:F_A2}
Given a quantum-commutative quantum-symmetric algebra
$A=(A,m_A, \mathrm{ev}_A, \mathrm{coev}_A)$ in $\mathcal{B}$, there
uniquely exists a braided, almost strict tensor functor
$F_A:\mathcal{T}\to\mathcal{B}$ which satisfies the equalities given in
\eqref{eq:braided-functor}.
\end{proposition}

\begin{proof}
For a tensor category $\mathcal{B}$ in general, a strict tensor category
$\mathcal{B}^{str}$ together with a strict tensor equivalence
$G:\mathcal{B}\to\mathcal{B}^{str}$ is constructed in~\cite[Sect. XI5]{Kassel95}.

Suppose we are in our special situation.
It follows by the construction of \cite[Sect.~XI5]{Kassel95} that for
every $n,m>0$, the unique canonical isomorphism
$A^{\otimes n}\otimes A^{\otimes m}\overset{\simeq}{\longrightarrow}A^{\otimes(n+m)}$
is sent by $G$ to the identity on $G(A)^{\otimes(n+m)}$.
We can apply Remark~\ref{rem:braided-alg} to
$(G(A),G(m_A),G(c_{A,A}^{\pm1}),G(\mathrm{ev}_A),G(\mathrm{coev}_A))$
in $\mathcal{B}^{str}$.
The result is translated via $G$ so that the same result as in the
remark holds true with ``a strict tensor functor $F_A$'' replaced with
``an almost strict tensor functor $F_A$.''
This implies the desired result.
\end{proof}

\begin{corollary} \label{cor:invariant}
Given a quantum-commutative quantum-symmetric algebra $A$ in a braided
tensor category, $F_A(H)$ gives an invariant for handlebody-links $H$,
which has values in the endomorphism monoid $\operatorname{End}(I)$.
\end{corollary} 

\begin{proof}
This is a direct consequence of Proposition~\ref{prop:F_A2}.
\end{proof}

\section{Quantum-commutative quantum-symmetric algebras in Yetter--Drinfeld modules}
\label{sec:YDmodule}

We will show that every finite-dimensional unimodular Hopf algebra over
a field, regarded as an algebra in a braided tensor category of
Yetter--Drinfeld modules, is a quantum-commutative quantum-symmetric
algebra, which is unital.

In what follows we work over a fixed base field $k$; the tensor products
$\otimes$ denote those for vector spaces over $k$.
Let $A$ be a Hopf algebra with the coproduct
$\Delta:A\to A\otimes A$, the counit $\varepsilon:A\to k$ and the
antipode $S:A\to A$.
For $\Delta$, we will use the following variant of the Sweedler
notation~\cite[Sect. 1.2]{Sweedler69}:
\begin{align*}
&\Delta(a)=a_{(1)}\otimes a_{(2)}, &
&\Delta(a_{(1)})\otimes a_{(2)}
=a_{(1)}\otimes a_{(2)}\otimes a_{(3)}
=a_{(1)}\otimes\Delta(a_{(2)}).
\end{align*}
For our purpose we may and we do assume that $A$ is finite-dimensional.
Then $A$ includes the one-dimensional subspaces
\begin{align*}
I_l(A)&:=\{\Lambda\in A\,|\,a\Lambda=\varepsilon(a)\Lambda,~a\in A\}, \\
\text{resp.,~}
I_r(A)&:=\{\Lambda\in A\,|\,\Lambda a=\varepsilon(a)\Lambda,~a\in A\}
\end{align*}
consisting of all \textit{left} and resp., \textit{right integrals};
see~\cite[Corollary 5.1.6]{Sweedler69}.
It possibly happen that $I_l(A)\neq I_r(A)$.

\begin{definition}[{\cite[Definition 2.1.1]{Montgomery93}}]
\label{def:unimodular}
$A$ is said to be \textit{unimodular}, if $I_l(A)=I_r(A)$.
\end{definition}

The assumption $\dim A<\infty$ ensures that the antipode $S$ is
bijective; see~\cite[Corollary 5.1.6]{Sweedler69}.

Given a left $A$-comodule $V$, we will write its structure, say
$\rho:V\to A\otimes V$, explicitly so that
\[ \rho(v)=v_{(-1)}\otimes v_{(0)};~
\text{cf.~\cite[Sect. 2.0]{Sweedler69}}. \]

Let ${}_A^A\mathcal{YD}$ denote the $k$-linear abelian, braided tensor
category of \textit{Yetter--Drinfeld modules} over $A$;
see~\cite[Definition 10.6.10]{Montgomery93}.
Such a module is by definition a left $A$-module $V$ given a left
$A$-comodule structure $\rho:V\to A\otimes V$ which satisfies
\[ \rho(av)=a_{(1)}v_{(-1)}S(a_{(3)})\otimes a_{(2)}v_{(0)}, ~~~ a\in A,v\in V. \]
Morphisms in ${}_A^A\mathcal{YD}$ are $A$-linear and $A$-colinear maps.
In ${}_A^A\mathcal{YD}$, the tensor product, the unit object
(that is $k$), and the associativity and unit constraints are the
obvious ones, being the same as those for left (co)modules.
The braiding is defined by
\[ c_{V,W}:V\otimes W\overset{\simeq}{\longrightarrow}W\otimes V, ~~~
c_{V,W}(v\otimes w)=v_{(-1)}w\otimes v_{(0)}, \]
whose inverse is given by
\[ c_{V,W}^{-1}(w\otimes v)=v_{(0)}\otimes S^{-1}(v_{(-1)})w. \]

As is well known, the braided tensor category ${}_A^A\mathcal{YD}$ thus
defined is naturally identified with that of left modules over the
quantum double $D(A)$;
see~\cite[Sect.\, 10.6]{Montgomery93} or \cite[Sect.\, IX.5]{Kassel95}.
Though the latter category might be more familiar, the former is more
suitable for our purpose.
In the following section we will treat with the quantum double $D(kG)$
of a finite group algebra $kG$, as an example of the present $A$.
  
We regard $A$ as a left $A$-module with respect to the conjugate action
$\triangleright$ defined by
\[ a\triangleright b=a_{(1)}bS(a_{(2)}), ~~~ a,b\in A. \]
We regard $A$ as a left $A$-comodule with respect to the coproduct
$\Delta:A\to A\otimes A$.

\begin{lemma} \label{lem:YD-alg}
We have $(A,\triangleright,\Delta)\in{}_A^A\mathcal{YD}$.
Moreover, this $A$, equipped with the original algebra structure, turns
into a unital algebra in ${}_A^A\mathcal{YD}$ which satisfies the
quantum-commutativity~\eqref{eq:q-commutativity}.
\end{lemma}

\begin{proof}
This is directly verified;
the quantum-commutativity follows from $(a_{(1)}\triangleright b)a_{(2)}=ab$.
%%% detail %%%
%\begin{align*}
%&\rho(a\triangleright b)=\Delta(a_{(1)}bS(a_{(2)}))
%=a_{(1)}b_{(1)}S(a_{(4)})\otimes a_{(2)}b_{(2)}S(a_{(3)})
%=a_{(1)}b_{(1)}S(a_{(3)})\otimes a_{(2)}\triangleright b_{(2)}, \\
%&(a_{(1)}\triangleright b)a_{(2)}=a_{(1)}bS(a_{(2)})a_{(3)}=a_{(1)}b\varepsilon(a_{(2)})=ab.
%\end{align*}
\end{proof}

Continue to suppose that $A$ is a finite-dimensional Hopf algebra.
The dual vector space $A^*=\operatorname{Hom}_k(A,k)$ of $A$ forms
naturally a Hopf algebra (see~\cite[Sect. 6.2]{Sweedler69}), so that we
have $I_l(A^*)$, $I_r(A^*)$.
We will use the following well-known fact;
see~Proposition~1 (e) and Corollary~1 of \cite{Radford94a}, for example.

\begin{proposition} \label{prop:integral-integral}
Let $\lambda$ be a non-zero left or right integral in $A^*$.
\begin{itemize}
\item[(1)] 
If $\Lambda$ is a left (resp., right) integral in $A$, then $S^{\pm1}(\Lambda)$ is a 
right (resp., left) integral in $A$, and $\lambda(\Lambda) = \lambda(S^{\pm1}(\Lambda))$. 
\item[(2)]
There exists uniquely a left or right integral $\Lambda$ in $A$ such that
$\lambda(\Lambda)=1$.
It follows that the evaluation map $I_l(A^*)\otimes I_l(A)\to k$ and the analogous ones are all linear
isomorphisms.
\end{itemize}
\end{proposition} 

Choose $0\neq\lambda\in I_l(A^*)$, and define a bilinear form on $A$ by
\[ \langle~,~\rangle_\lambda:A\times A\to k,
\quad \langle a,b\rangle_\lambda=\lambda(ab). \]
The following is well known; see~\cite[Theorem 2.1.3]{Montgomery93}. 

\begin{proposition}\label{prop:form}
This bilinear form is non-degenerate.
\end{proposition}

Choose bases $(\alpha_i)$, $(\beta_i)$ of $A$ which are dual to each
other with respect to $\langle~,~\rangle_\lambda$, so that
$\langle\alpha_i,\beta_j\rangle_\lambda=\delta_{ij}$.
Set
\[ U_\lambda=\sum_i\beta_i\otimes\alpha_i\in A\otimes A. \]
This element is characterized by the property that 
\begin{equation}\label{eq:Ulambda}
\sum_i \beta_i~\langle \alpha_i, a\rangle_{\lambda} = a
~~\text{or}~~\sum_i \langle a, \beta_i\rangle_{\lambda}~\alpha_i = a
\end{equation}
for all $a\in A$. 

\begin{lemma}\label{lem:Ulambda} 
Let $0 \ne \lambda \in I_l(A^*)$ as above. 
\begin{itemize}
\item[(1)] Let $\Lambda$ be a unique right integral in $A$ such that $\lambda(\Lambda)=1$. Then 
\[ U_{\lambda} = S(\Lambda_{(1)})\otimes \Lambda_{(2)}. \] 
\item[(2)] Let $\Lambda$ be a unique left integral in $A$ such that $\lambda(\Lambda)=1$. Then 
\[ U_{\lambda} = \Lambda_{(2)}\otimes S^{-1}(\Lambda_{(1)}). \]
\end{itemize}
\end{lemma}

\begin{proof}
\begin{itemize}
\item[(1)]
Since $\Lambda \in I_r(A)$, we have
\begin{equation} \label{eq:right-integral}
\Lambda_{(1)}S^{-1}(a)\otimes\Lambda_{(2)}=\Lambda_{(1)}\otimes\Lambda_{(2)}a,~~a \in A.
\end{equation}
To see that the element $S(\Lambda_{(1)})\otimes\Lambda_{(2)}$ satisfies the first equation of \eqref{eq:Ulambda},
we have to show that for any $a\in A$, $f\in A^*$,
\[ f(S(\Lambda_{(1)}))\, \lambda(\Lambda_{(2)}a) = f(a). \]
From \eqref{eq:right-integral} and $\lambda \in I_l(A^*)$,
we see that this left-hand side equals
\[ S^*(f \leftharpoonup a)(\Lambda_{(1)}) \, \lambda(\Lambda_{(2)})
=S^*(f\leftharpoonup a)(1)\, \lambda(\Lambda) =f(a), \]
as desired, where $f\leftharpoonup a$ is defined by $(f\leftharpoonup a)(b) = f(ab)$, $b \in A$.
\item[(2)]
If $\Lambda \in I_l(A)$ with $\lambda(\Lambda)=1$,
then $S^{-1}(\Lambda)\in I_r(A)$ with $\lambda(S^{-1}(\Lambda))=1$,
by Proposition~\ref{prop:integral-integral} (1).
Part 1 applied to $S^{-1}(\Lambda)$ shows Part 2.
\end{itemize}
\end{proof}

To continue our construction we define linear maps,
\begin{align*}
&\mathrm{ev}_A:A\otimes A\to k, &
&\mathrm{ev}_A(a\otimes b)=\langle a,b\rangle_\lambda, \\
&\mathrm{coev}_A:k\to A\otimes A, &
&\mathrm{coev}_A(1)=U_\lambda.
\end{align*}
Let $m_A:A\otimes A\to A$ denote the product on $A$.
Then we have
\begin{align}
\mathrm{ev}_A=\lambda\circ m_A. \label{eq:ev}
\end{align}
Obviously, $\mathrm{ev}_A$, $\mathrm{coev}_A$ defined above satisfy the
selfduality~\eqref{eq:selfduality} and the Frobenius
property~\eqref{eq:Frobenius}.

\begin{proposition} \label{prop:Hopf-alg-as-q-sym-alg}
Assume that $A$ is unimodular.
\begin{itemize}
\item[(1)]
$\mathrm{ev}_A$ and $\mathrm{coev}_A$ defined above are both morphisms
in ${}_A^A\mathcal{YD}$.
\item[(2)]
The object $A=(A,\triangleright,\Delta)$ in ${}_A^A\mathcal{YD}$,
equipped with $m_A$, $\mathrm{ev}_A$, $\mathrm{coev}_A$, is a
quantum-commutative quantum-symmetric algebra in ${}_A^A\mathcal{YD}$.
\end{itemize}
\end{proposition}

\begin{proof}
\begin{itemize}
\item[(1)]
First, we show, without the unimodularity assumption, that
$\mathrm{ev}_A$ and $\mathrm{coev}_A$ are $A$-colinear.
Since $\lambda$, regarded as a linear map $A \to k$, is left
$A$-colinear, it follows by \eqref{eq:ev} that $\mathrm{ev}_A$ is $A$-colinear, since $m_A$
is obviously $A$-colinear. 
%%% detail %%%
%\begin{align*}
%\lambda\in I_l(A^*)
%&\Rightarrow f\lambda=\varepsilon(f)\lambda, ~~~ f\in A^*=\operatorname{Hom}_k(A,k) \\
%&\Rightarrow f(a_{(1)})\lambda(a_{(2)})=(f\lambda)(a)
%=\varepsilon(f)\lambda(a)=f(1)\lambda(a), ~~~ f\in A^*,a\in A \\
%&\Rightarrow a_{(1)}\lambda(a_{(2)})=1_A\lambda(a) \\
%&\Rightarrow\text{$\lambda$ is $A$-colinear.}
%\end{align*}
For $\mathrm{coev}_A$, let us use the expression of $U_{\lambda}$ given by Lemma~\ref{lem:Ulambda} (1). 
Then the $A$-colinearity of $\mathrm{coev}_A$ follows since we see
\[ S(\Lambda_{(1)})_{(1)}(\Lambda_{(2)})_{(1)} \otimes 
S(\Lambda_{(1)})_{(2)} \otimes (\Lambda_{(2)})_{(2)} = 1 \otimes U_{\lambda}. \]

To show the remaining $A$-linearity, assume that $A$ is unimodular.
Let $(A^{op})^{cop}$ denote the Hopf algebra $A$ with the opposite
product and coproduct; it has the same antipode as $A$, and our
$\lambda$ is a right integral in its dual Hopf algebra.
Apply the equality (a) of \cite[Theorem~3]{Radford94a} to this
$(A^{op})^{cop}$.
Since the unimodularity assumption implies that the $\alpha$ in that
equality equals $\varepsilon$, it follows that
\begin{equation}\label{eq:commutator}
\lambda(ab)=\lambda(bS^2(a)), ~~~ a,b\in A.
\end{equation}
Since the product $m_A$ is obviously $A$-linear, it follows by
\eqref{eq:ev} that in order to prove the $A$-linearity of
$\mathrm{ev}_A$, it suffices to see that $\lambda:A\to k$ is $A$-linear.
In fact, this holds true, since we see from \eqref{eq:commutator} that
for $a,b\in A$,
\begin{equation*} 
\lambda(a_{(1)}bS(a_{(2)}))
=\lambda(bS(a_{(2)})S^2(a_{(1)}))
=\varepsilon(a)\lambda(b).
\end{equation*}
The $A$-linearity of $\mathrm{coev}_A$ will follow if one sees, using the same expression
of $U_{\lambda}$ as above, that for every $a \in A$, 
\[ a_{(1)}\rhd S(\Lambda_{(1)}) \otimes a_{(2)} \rhd \Lambda_{(2)} 
= \varepsilon(a)\, U_{\lambda}.  \]
Use \eqref{eq:right-integral} and the analogous equation 
\[ a \Lambda_{(1)} \otimes \Lambda_{(2)} = \Lambda_{(1)} \otimes S^{-1}(a)\Lambda_{(2)}, \] 
which holds since $\Lambda \in I_l(A)$. 
Then we see that the left-hand side of the desired equation equals
\begin{align*}
&a_{(1)}S(\Lambda_{(1)})S(a_{(2)}) \otimes a_{(3)}\Lambda_{(2)}S(a_{(4)}) \\
&=S(a_{(2)} \Lambda_{(1)} S^{-1}(a_{(1)})) \otimes a_{(3)}\Lambda_{(2)} S(a_{(4)}) \\
&=S(\Lambda_{(1)}) \otimes a_{(3)}S^{-1}(a_{(2)})\Lambda_{(2)} a_{(1)}S(a_{(4)}),
\end{align*}
which is seen to equal the right-hand side.
\item[(2)]
It remains to verify the quantum-symmetry.
By~\eqref{eq:ev}, this desired property follows from the
quantum-commutativity which was verified by Lemma~\ref{lem:YD-alg}.
\end{itemize}
\end{proof}

\begin{remark}\label{rem:biGalois}
The construction above is generalized as follows.
Suppose that $A$, $L$ are finite-dimensional Hopf algebras, and that $B$
is an $(L,A)$-biGalois object, that is, an $(L,A)$-bicomodule algebra
which is a Galois extension \cite[Definition 8.1.1]{Montgomery93} over
$k$ on both sides.
Choose $0\neq\lambda\in I_l(A^*)$, and define
$\langle b,c\rangle_\lambda=b_{(0)}c_{(0)}\lambda(b_{(1)}c_{(1)})$
for $b,c\in B$, where $b\mapsto b_{(0)}\otimes b_{(1)}$ denotes the
right $A$-comodule structure on $B$.
Then one can prove that this last defines indeed a bilinear form
$\langle~,~\rangle_\lambda:B\times B\to k$
which is non-degenerate.
By the same procedure as above, we see that if $A$ is unimodular, then
$B$ turns into a quantum-commutative quantum-symmetric algebra in
${}_L^L\mathcal{YD}$, where the left $L$-module structure on $B$ is
given by the so-called Miyashita-Ulbrich action.
This construction applied to $A$, which is regarded naturally as an
$(A,A)$-biGalois object, produces the quantum-commutative
quantum-symmetric algebra in ${}_A^A\mathcal{YD}$ given by the last
proposition.

However, we have a natural equivalence ${}_A^A\mathcal{YD}\approx{}_L^L\mathcal{YD}$
of braided tensor categories (see~\cite[Proposition 5.1]{Masuoka08}, for example),
under which $A$ and $B$ correspond to each other, so that the associated
braided tensor functors $F_A$, $F_B$ are identified via the equivalence.
Therefore, we may restrict ourselves to Hopf algebras, without working
with biGalois objects.
\end{remark}

\section{Invariants derived from unimodular Hopf algebras}

Let $A$ be a finite-dimensional unimodular Hopf algebra, and choose $0\neq\lambda\in I_l(A^*)$.
By Proposition~\ref{prop:Hopf-alg-as-q-sym-alg},
$A$ turns into a quantum-commutative quantum-symmetric algebra $A$ in ${}_A^A\mathcal{YD}$.
By Corollary~\ref{cor:invariant} this $A$ gives an invariant $F_A(H)$ for handlebody-links $H$,
which has values in $k$ since the endomorphism ring $\mathrm{End}(k)$
in ${}_A^A\mathcal{YD}$ coincides with $k$.
The invariant $F_A(H)$ depends on choice of $\lambda$,
but we will not indicate it within the notation except in the following.

\begin{remark}\label{rem:choice-of-lambda}
Let us write here $F_{A,\lambda}(H)$ for $F_A(H)$, indicating $\lambda$.
Let $0\ne c\in k$.
If we replace $\lambda$ with $c\lambda$, then $\mathrm{ev}_A$
(resp.,~$\mathrm{coev}_A$) is replaced by its scalar multiple by
$c$ (resp.,~by $c^{-1}$).
Therefore, we have
\[ F_{A,c\lambda}(H)=c^{\#\cap-\#\cup} F_{A,\lambda}(H), \]
where $\#\cap$, $\#\cup$ respectively denote the numbers of $\cap$,
$\cup$ in $H$.
\end{remark}

Here is the simplest example of computations.

\begin{example} \label{ex:sampleO}
Let $O$ be the trivial genus $1$ handlebody-knot, which is represented
by the trivial knot.
Then,
\[ F_A(O)= \operatorname{Trace}S^2, \]
the trace of the linear endomorphism $S^2 = S \circ S$ of $A$. In particular, $F_A(O)
= (\operatorname{dim}A)1$, if $S$ is an involution, that is, $S^2 = \mathrm{id}_A$. 
To prove the formula above, we use the expression of $U_{\lambda}$ given by Lemma~\ref{lem:Ulambda} (1).
Then it follows by Eq.~(3) of \cite{LarsonRadford87} that
\[ F_A(O) = \lambda(S(\Lambda_{(1)})\, \Lambda_{(2)}) 
= \lambda(1)\, \varepsilon(\Lambda)=\operatorname{Trace}S^2. \]  
\end{example}

We should carefully choose $A$ for the invariant, as is seen from the
following proposition, whose proof will be postponed for a moment.

\begin{proposition}\label{prop:zero-invariant}
If $A$ is not cosemisimple, then $F_A(H)=0$ for any handlebody-link $H$.
\end{proposition}

We say that a finite-dimensional Hopf algebra $A$ is \emph{cosemisimple}
if $A^*$ is semisimple as an algebra.
We recall the following fundamental results on (co)semisimplicity; see
\cite[Theorem~5.1.8]{Sweedler69} for (1), 
\cite[Theorem~4]{LarsonRadford87}, \cite[Theorem~3.3]{LarsonRadford88}
for (2), and \cite[Corolllary~3.2]{EtingofGelaki98} for (3). 

\begin{theorem}\label{thm:cosemisimple}
Let $A$ be a finite-dimensional Hopf algebra.
\begin{itemize}
\item[(1)] (Sweedler)
The following are equivalent:
\begin{itemize}
\item[(a)] $A$ is cosemisimple;
\item[(b)] There exists a left or right integral $\lambda$ in $A^*$ such that $\lambda(1)=1$;
\item[(c)] There exists a left and right integral $\lambda$ in $A^*$ such that $\lambda(1)=1$.
\end{itemize}
\item[]
In particular, if $A$ is cosemisimple, then $A^*$ is unimodular. 
\item[(2)] (Larson--Radford)
Assume $\operatorname{char}k=0$.
Then the following are equivalent:
\begin{itemize}
\item[(a)] $A$ is cosemisimple;
\item[(d)] $A$ is semisimple;
\item[(e)] The antipode $S$ is an involution.
\end{itemize}
\item[(3)] (Etingof--Gelaki)
Assume $\operatorname{char}k>0$.
Then the following are equivalent:
\begin{itemize}
\item[(f)] $A$ is semisimple and cosemisimple;
\item[(g)] $S$ is an involution, and $\operatorname{char}k$ does not divide $\dim A$.
\end{itemize}
\end{itemize}
\end{theorem}

We remark that $\lambda$ such as in (b), (c) of Part 1 above is unique.
As the dual result of Part 1, a finite-dimensional semisimple Hopf
algebra is unimodular.
It follows from Part 2 that if $\operatorname{char}k=0$, then a
finite-dimensional cosemisimple Hopf algebra is necessarily unimodular.
Whether the same statement holds true in positive characteristic seems
an open problem.

Our proof of Proposition~\ref{prop:zero-invariant} is based on the following fact.
 
\begin{lemma} \label{lem:ktoA}
Let $A$ be a finite-dimensional Hopf algebra.
The vector space ${}_A^A\mathcal{YD}(k,A)$ of all morphisms
$\phi:k\to A=(A,\triangleright,\Delta)$ in ${}_A^A\mathcal{YD}$ is
isomorphic, via $\phi\mapsto\phi(1)$, to the sub-vector space $k1$ in
$A$.
\end{lemma}

\begin{proof}
Note that given an object $V$ in ${}_A^A\mathcal{YD}$, the vector space
${}_A^A\mathcal{YD}(k,V)$ of all morphisms $\phi:k\to V$ is isomorphic,
via $\phi\mapsto\phi(1)$, to the sub-vector space of $V$ which consists
of the elements $v$ such that
\begin{align}
&av=\varepsilon(a)v, ~~~ a\in A
&\text{and}&
&\rho(v)=1\otimes v.
\label{eq:invariant}
\end{align}
Suppose $V =A$. Then the second condition of \eqref{eq:invariant} is
equivalent to $v\in k1$, which implies the first condition.
This proves the lemma.
\end{proof}

For the following proof and for later use, let $H$ be a handlebody-link.
Choose arbitrarily one from the top handlebody-tangles $\cap$ in $H$, and replace it by $\curlywedge$.
Let $H^{\curlywedge}$ denote the resulting handlebody-tangle; see Figure~\ref{fig:(1,0)-tangle}.
We thus have $s(H^{\curlywedge})=0$, $b(H^{\curlywedge})=1$.
We call $H^{\curlywedge}$ a \emph{handlebody-tangle horned to $H$};
this varies according to choice of the top $\cap$.

\begin{proof}[Proof of Proposition~\ref{prop:zero-invariant}]
Let $H$, $H^{\curlywedge}$ be as above.
Lemma~\ref{lem:ktoA} shows that $F_A(H^{\curlywedge})$ has values in $k1$ ($\subset A$).
Since $I_{l}(A^*)$ is spanned by $\lambda$,
Theorem~\ref{thm:cosemisimple} (1) shows that the non-cosemisimplicity
assumption is equivalent to the condition that $\lambda$ vanishes on $k1$,
which implies that $F_A(H)=\lambda\circ F_A(H^{\curlywedge})(1)=0$,
since $F_A(\cap)=\mathrm{ev}_A=\lambda\circ m_A$ by \eqref{eq:ev}.
\end{proof}

\begin{figure}
\mbox{}\hfill\hfill
\begin{picture}(40,50)
\qbezier(0,10)(20,10)(40,10)
\qbezier(0,30)(20,30)(40,30)
\qbezier(0,10)(0,20)(0,30)
\qbezier(40,10)(40,20)(40,30)
\qbezier(10,30)(10,40)(20,40)
\qbezier(30,30)(30,40)(20,40)
\put(20,0){\makebox(0,0){$H$}}
\end{picture}
\hfill
\begin{picture}(40,50)
\qbezier(0,10)(20,10)(40,10)
\qbezier(0,30)(20,30)(40,30)
\qbezier(0,10)(0,20)(0,30)
\qbezier(40,10)(40,20)(40,30)
\qbezier(10,30)(15,35)(20,40)
\qbezier(30,30)(25,35)(20,40)
\qbezier(20,40)(20,45)(20,50)
\put(20,0){\makebox(0,0){$H^{\curlywedge}$}}
\end{picture}
\hfill\hfill\mbox{}
\caption{}
\label{fig:(1,0)-tangle}
\end{figure}

By modifying $F_A(H)$, we wish to obtain some meaningful invariant of handlebody-links $H$,
when $A$ is not necessarily cosemisimple.
Let $A$ be a finite-dimensional unimodular Hopf algebra,
and choose $0\ne\lambda\in I_l(A^*)$.
Let $Z(A)$ denote the center of $A$.
Assume that
\begin{equation}\label{eq:assumption}
\lambda(z)=\lambda(S(z)), ~~ z\in Z(A).
\end{equation}
This assumption is independent of choice of $\lambda$,
and is satisfied if $A^*$ is unimodular, since then $\lambda = \lambda \circ S$,
as is seen from Proposition~\ref{prop:integral-integral}.
In particular, it is satisfied if $A$ is cosemisimple; see Theorem~\ref{thm:cosemisimple} (1).
There are known examples of finite-dimensional cocommutative Hopf algebras which are not unimodular;
see \cite[p.238]{Montgomery93}, for example.
Their dual Hopf algebras are examples of finite-dimensional
unimodular Hopf algebras which do not satisfy \eqref{eq:assumption}.

\begin{definition} \label{def:modified_invariant}
Let $A$, $\lambda$ be as above.
For a handlebody-link $H$, we define a scalar $v_A(H)$ in $k$ by
\[ v_A(H)=\varepsilon\circ F_A(H^{\curlywedge})(1), \]
where $H^{\curlywedge}$ is a handlebody-tangle horned to $H$.
Notice from Lemma~\ref{lem:ktoA} that $v_A(H)=F_A(H^{\curlywedge})$ in $k$.
\end{definition}

We have to show that the value $F_A(H^{\curlywedge})(1)$ is independent of
choice of the top $\cap$ to be replaced by $\curlywedge$.
This will be proved below Lemma~\ref{lem:double-horns}.

\begin{remark} \label{rem:v_for_cosemisimple}
\begin{itemize}
\item[(1)]
Suppose that $A$ is cosemisimple.
Then by Theorem \ref{thm:cosemisimple} (1), $\lambda$ can be chosen so that $\lambda(1)=1$.
In this case we have $v_A(H) = F_A(H)$ for every handlebody-link $H$, since
\[ v_A(H)=\varepsilon\circ F_A(H^{\curlywedge})(1)=\lambda\circ F_A(H^{\curlywedge})(1)=F_A(H). \]
\item[(2)]
If $A$ is not cosemisimple, we do not have any canonical choice of $\lambda$ as above.
We see from Remark~\ref{rem:choice-of-lambda} that if $\lambda$ is replaced by $c\lambda$ with $0\ne c \in k$,
then $v_A(H)$ changes by the scalar multiple by $c^{\#\cap-\#\cup-1}$.
\end{itemize}
\end{remark} 

\begin{convention}
Taking Part 1 above into account,
we will hereafter choose $\lambda$ so that $\lambda(1)=1$ if $A$ is cosemisimple.
\end{convention}

\begin{lemma}\label{lem:ktoAA}
Let $A$ be a finite-dimensional Hopf algebra.
Then the vector space ${}_A^A\mathcal{YD}(k,A\otimes A)$ of all
morphisms $\phi:k\to A\otimes A$ in ${}_A^A\mathcal{YD}$, where
$A\otimes A$ is the tensor product of two copies of
$A=(A,\triangleright,\Delta)$, is isomorphic, via
$\phi\mapsto \phi(1)$, to the sub-vector space of $A\otimes A$
consisting of the elements $S(z_{(1)})\otimes z_{(2)}$, where $z$ is an
arbitrary element in the center $Z(A)$ of $A$.
\end{lemma}

\begin{proof}
Set $V=A\otimes A$.
Give to the same vector space $A \otimes A$,
an alternative structure of a Yetter--Derinfeld module by defining
\[a(b\otimes c):=a_{(1)}bS(a_{(4)})\otimes a_{(2)}cS(a_{(3)}),
~~\rho(b\otimes c) = b_{(1)}\otimes(b_{(2)}\otimes c), \]
where $a\in A$, $b\otimes c\in A\otimes A$.
Let $V'$ denote the thus defined object.
We see that $b\otimes c\mapsto bc_{(1)}\otimes c_{(2)}$ gives an isomorphism
$V\overset{\simeq}{\longrightarrow}V'$ in ${}^A_A\mathcal{YD}$,
whose inverse is given by $b\otimes c\mapsto bS(c_{(1)})\otimes c_{(2)}$. 
As is easily seen, the elements $1\otimes z$, $z\in Z(A)$ are precisely
those elements in $V'$ which satisfies the conditions \eqref{eq:invariant}.
It follows that the elements $S(z_{(1)})\otimes z_{(2)}$,
$z\in Z(A)$ are precisely those which satisfies the same conditions.
The proof of Lemma~\ref{lem:ktoA} shows the desired result.
\end{proof}

For the rest of this section, let $A$ be a finite-dimensional unimodular Hopf algebra,
and choose $0\neq\lambda\in I_l(A^*)$ (so that $\lambda(1)=1$ if $A$ is cosemisimple). 

\begin{lemma} \label{lem:double-horns}
Assume \eqref{eq:assumption}.
For any handlebody-tangle $T$ such that $s(T)=0$, $b(T)=2$, we have
\[ (\varepsilon\otimes\lambda)\circ F_A(T)
=(\lambda\otimes\varepsilon)\circ F_A(T). \]
\end{lemma}

\begin{proof}
By Lemma~\ref{lem:ktoAA}, for a morphism $F_A(T):k\to A\otimes A$ in
${}_A^A\mathcal{YD}$, we have $F_A(T)(1)=S(z_{(1)})\otimes z_{(2)}$ for
some $z\in Z(A)$.
The desired result will follow if we see that
$\lambda(S(z_{(1)}))z_{(2)}=S(z_{(1)})\lambda(z_{(2)})$, or
equivalently,
\[ \lambda(S(z_{(1)}))f(z_{(2)})
%%% detail %%%
%=f(1)\lambda(S(z))=f(1)\lambda(z)=f(S(1))\lambda(z)
=f(S(z_{(1)}))\lambda(z_{(2)}), ~~~ f\in A^*. \]
Since $\lambda\in I_l(A^*)$ and $\lambda\circ S=S^*(\lambda)\in I_r(A^*)$,
we see that the assumption \eqref{eq:assumption} ensures this last desired condition.
\end{proof}

The desired independency of the value $v_A(H)$ follows
since Lemmas~\ref{lem:ktoA} and \ref{lem:double-horns} show
\begin{center}
\begin{picture}(180,60)
\qbezier(0,10)(35,10)(70,10)
\qbezier(0,30)(35,30)(70,30)
\qbezier(0,10)(0,20)(0,30)
\qbezier(70,10)(70,20)(70,30)
\qbezier(10,30)(15,35)(20,40)
\qbezier(30,30)(25,35)(20,40)
\qbezier(20,40)(20,45)(20,50)
\qbezier(40,30)(40,40)(50,40)
\qbezier(60,30)(60,40)(50,40)
\put(90,20){\makebox(0,0){$=$}}
\qbezier(110,10)(145,10)(180,10)
\qbezier(110,30)(145,30)(180,30)
\qbezier(110,10)(110,20)(110,30)
\qbezier(180,10)(180,20)(180,30)
\qbezier(120,30)(120,40)(130,40)
\qbezier(140,30)(140,40)(130,40)
\qbezier(150,30)(155,35)(160,40)
\qbezier(170,30)(165,35)(160,40)
\qbezier(160,40)(160,45)(160,50)
\put(195,15){\makebox(0,0){.}}
\end{picture}
\end{center}
The same idea as proving Lemma~\ref{lem:double-horns} shows the following as well.

\begin{proposition} \label{prop:connected-sum}
Assume \eqref{eq:assumption}.
Given two handlebody-links $H_i$ $(\subset B_i)$, $i=1,2$, contained in disjoint balls $B_i$,
let $H_1\#H_2$ denote the handlebody-link obtained by attaching them by a $1$-handle.
Then we have
\[ v_A(H_1\#H_2)=v_A(H_1)v_A(H_2). \]
\end{proposition}

\begin{proof}
For $i =1,2$, let $H_i^{\curlywedge}$ be a handlebody-tangle horned to $H_i$.
Then $H_1\#H_2=\mcap\circ(H_1^{\curlywedge}\otimes H_2^{\curlywedge})$.
Since $\varepsilon:A\to k$ is an algebra map, we have
\[ v_A(H_1\#H_2)=\varepsilon\circ m_A(F_A(H_1^{\curlywedge})(1)
\otimes F_A(H_2^{\curlywedge})(1))=v_A(H_1)v_A(H_2). \]
\end{proof}

Given a handlebody-link (or more generally, a handlebody-tangle) $H$,
let $H^*$ denote its mirror image.
Let us evaluate $v_A(H^*)$.
Let $A^{op}$ denote the Hopf algebra $A$ with the opposite product;
it has $S^{-1}$ as its antipode.
We can and we do choose the same $\lambda$ as the original one
as a non-zero left integral in $(A^{op})^*$.

\begin{proposition} \label{prop:mirror-image-knot}
For a handlebody-link $H$, we have
\[ v_A(H^*) = v_{A^{op}}(H). \]
\end{proposition}

We prove this in a generalized situation.
Let $\tau_1:A\overset{\simeq}{\longrightarrow}A^{op}$, $\tau_1(a)=a^{op}$
denote the canonical linear isomorphism, so that $a^{op}b^{op}=(ba)^{op}$,
where $a,b\in A$.
Let $\tau_0:k\to k$ be the identity map.
For $n>1$, let
$\tau_n:A^{\otimes n}\overset{\simeq}{\longrightarrow}(A^{op})^{\otimes n}$
be the linear isomorphism defined by
\[ \tau_n(a_1\otimes a_2\otimes\dots\otimes a_n)
=a_n^{op}\otimes\dots\otimes a_2^{op}\otimes a_1^{op}, ~~~ a_i\in A. \]

\begin{proposition} \label{prop:mirror-image-tangle}
Let $T$ be a handlebody-tangle such that $s(T)=m$, $b(T)=n$.
Then we have
\[ \tau_n\circ F_A(T^*)=F_{A^{op}}(T)\circ\tau_m. \]
\end{proposition}

\begin{proof}
By Proposition~\ref{prop:new-relations}, we may suppose that $T$ is one
of the five tangles listed there.

Suppose $T=X$.
Then $T^*=\overline{X}$.
An element $a\otimes b$ in $A\otimes A$ is sent by
$\tau_2\circ F_A(\overline{X})$ to
$(S^{-1}(b_{(2)})ab_{(1)})^{op}\otimes b_{(3)}^{op}$, while it is sent
by $F_{A^{op}}(X)\circ \tau_2$ to
$b_{(1)}^{op}a^{op}S^{-1}(b_{(2)}^{op})\otimes b_{(3)}^{op}$.
Obviously, the two results coincide.

Suppose $T=\cup$.
Then $T^*=\cup$.
If $(\alpha_i)$, $(\beta_i)$ are the dual bases with respect to
$\langle~,~\rangle_\lambda:A\times A\to k$, that is,
$\lambda(\alpha_i\beta_j)=\delta_{ij}$, then $(\beta_i^{op})$,
$(\alpha_i^{op})$ are the dual bases with respect to
$\langle~,~\rangle_\lambda:A^{op}\times A^{op}\to k$.
This implies that
$\tau_2\circ F_A(\cup)(1)=F_{A^{op}}(\cup)\circ\tau_0(1)$.

Similarly, the desired results follow in the remaining three cases.
\end{proof}

Since $\varepsilon\circ\tau_1=\varepsilon$,
Proposition~\ref{prop:mirror-image-knot} follows from
Proposition~\ref{prop:mirror-image-tangle} in the special situation when
$m=0,n=1$.

\section{First examples of unimodular Hopf algebras}

We raise below three examples of finite-dimensional unimodular Hopf algebras $A$,
giving their data needed to compute the invariants $v_A(H)$.
The duals $A^*$ are all unimodular, and the antipodes $S$ of $A$ are involutions.
It follows by Theorem~\ref{thm:cosemisimple} that if $\operatorname{char}k \nmid \operatorname{dim}A$,
then $A$ is semisimple and cosemisimple.
We choose two-sided integrals $\lambda$ in $A^*$ and $\Lambda$ in $A$ such that $\lambda(\Lambda)=1$.

\begin{example}\label{ex:Hopf-alg1}
Let $A=kG$ be the group algebra, where $G$ is a finite group.
The Hopf algebra structure is given by
\begin{align*}
&\Delta(g)=g\otimes g, &
&\varepsilon(g)=1, &
&S(g)=g^{-1}
\end{align*}
where $g\in G$.
We have
\begin{align*}
&c_{A,A}(g\otimes h)=ghg^{-1}\otimes g, &
&\lambda(g)=\delta_{1,g}, &
&\Lambda=\sum_{g\in G}g,
\end{align*}
\begin{align*}
&\mathrm{ev}_A(g \otimes h)=\delta_{1,gh}, &
&\mathrm{coev}_A(1)=\sum_{g\in G}g\otimes g^{-1},
\end{align*}
where $g,h\in G$.

Note $\lambda(1)=1$ and that this $A$ is cosemisimple.
By Maschke's Theorem, $A$ is semisimple if and only if
$\operatorname{char}k$ does not divide the order $|G|$ of $G$.

We remark that if $\operatorname{char}k=0$,
then the invariant $v_A(H)$ coincides with the number of the homomorphisms
from the fundamental group of the exterior of a handlebody-knot $H$ to the group $G$.
For, when we regard the value of the invariant as a state sum,
each state corresponds to the $G$-coloring of the diagram.
\end{example}

\begin{example}\label{ex:Hopf-alg2}
Let $A=D(kG)$ be the quantum double of $kG$, where $G$ is a finite
group.
Note that the dual Hopf algebra $(kG)^*$ of $kG$ is spanned by those
orthogonal idempotents $e_g$, $g\in G$, which are defined by
$e_g(h)=\delta_{g,h}$, where $g,h\in G$.
As a coalgebra, $A=kG\otimes(kG)^*$, and so
\begin{align*}
&\Delta(a\otimes e_g)=\sum_{h\in G}(a\otimes h)\otimes(a\otimes h^{-1}g), &
&\varepsilon(a\otimes e_g)=\delta_{1,g},
\end{align*}
where $a,g\in G$.
The product and the antipode on $A$ are given by
\[ (a\otimes e_g)(b\otimes e_h)=\delta_{g,bhb^{-1}}ab\otimes e_h, \]
\[ S(a\otimes e_g)=(1\otimes e_{g^{-1}})(a^{-1}\otimes 1)
=a^{-1}\otimes e_{ag^{-1}a^{-1}}, \]
where $a,b,g,h\in G$.
The unit equals $1\otimes1$.
The remaining data are given by
\[ c_{A,A}((a\otimes e_g)\otimes(b\otimes e_h))
=(aba^{-1}\otimes e_{aha^{-1}})\otimes(a\otimes e_{hbh^{-1}b^{-1}g}), \]
\begin{align*}
&\lambda(a\otimes e_g)=\delta_{1,a}, &
&\Lambda=\sum_{a\in G}a\otimes e_1,
\end{align*}
\[ \mathrm{ev}_A((a\otimes e_g)\otimes(b\otimes e_h))
=\delta_{1,a}\delta_{g,bhb^{-1}}, \]
\[ \mathrm{coev}_A(1\otimes1)
=\sum_{a,g\in G}(a\otimes e_g)\otimes(a^{-1}\otimes e_{aga^{-1}}). \]

Note $\lambda(1)=|G|1$.
This $A$ is semisimple if and only if it is cosemisimple
if and only if $\operatorname{char}k \nmid |G|$.
If these equivalent conditions hold, we should replace the integrals $\lambda$,
$\Lambda$ above with $|G|^{-1}\lambda$, $|G|\Lambda$, respectively.
\end{example}

\begin{example}\label{ex:Hopf-alg3}
Assume that the characteristic $\operatorname{char}k$ of $k$ is not 2.
Fix an integer $m>2$.
Let $A=\mathcal{B}_{4m}$ be the Hopf algebra as defined
by~\cite[Definition 3.3(2)]{Masuoka00}.
As an algebra this is generated by three elements, $a,t,z$, and is
defined by the relations
\begin{align*}
&a^2=t^2=1, &&ta=at, &&za=az, &&z^m=a, &&zt=tz^{-1}.
\end{align*}
Here we have re-chosen the generators $s_{\pm1}$ given in~\cite[Definition 3.3(2)]{Masuoka00}
so that $t=s_+$, $z=s_+s_-$, as in~\cite[Page 203, line --3]{Masuoka00}.
Note $z^{-1}=az^{m-1}$.
Set $e_0=(1/2)(1+a)$, $e_1=(1/2)(1-a)$; these are central idempotents in
$A$ such that $e_0e_1=0$, $e_0+e_1=1$.
The structure on $A$ is given by
\begin{align*}
\Delta(a)&=a\otimes a, &\varepsilon(a)&=1, &S(a)&=a, \\
\Delta(t)&=t\otimes e_0t+tz\otimes e_1t, &\varepsilon(t)&=1, &S(t)&=t(e_0+e_1z), \\
\Delta(z)&=z\otimes e_0z+z^{-1}\otimes e_1z, &\varepsilon(z)&=1, &S(z)&=e_0z^{-1}+e_1z.
\end{align*}
This $A$ has $(a^it^jz^k)_{0\le i,j<2,0\le k<m}$ as a basis, so that
$\dim A=4m$.
Note that $(e_it^jz^k)_{0\le i,j<2,0\le k<m}$ is another basis of $A$.
Let $0\le i,j,p,q<2$, $0\le k,r<m$.
Set
\[ d(i,j,k,p,q,r)=(-1)^j\{r-(-1)^{i+p}(2k-j)q\}+jq. \]
Then we have
\[ c_{A,A}(e_it^jz^k\otimes e_pt^qz^r)
=e_pt^q z^{d(i,j,k,p,q,r)}\otimes e_it^jz^k. \]
Note that if $q=0$ in particular, then
\[ c_{A,A}(e_it^jz^k\otimes e_pz^r)=e_pz^{(-1)^jr}\otimes e_it^jz^k. \]
The remaining data are given by
\begin{align*}
&\lambda(a^it^jz^k)=\delta_{(i,j,k),(0,0,0)}, &
&\Lambda=(1+a)(1+t)(1+z+\dots+z^{m-2}+az^{m-1}),
\end{align*}
\[ \mathrm{ev}_A(a^it^jz^k\otimes a^pt^qz^r)
=\delta_{(i,j,k),(p,q,r)}(\delta_{(j,k),(0,0)}+\delta_{j,1})
+\delta_{(i,j,k),(1-p,q,m-r)}\delta_{j,0}, \]
\[ \mathrm{coev}_A(1)
=\sum_{0\le i<2}a^i\otimes a^i+\sum_{\substack{0\le i<2\\ 0\le k<m}}a^itz^k\otimes a^itz^k
+\sum_{\substack{0\le i<2\\0\le k<m}}a^iz^k\otimes a^{i+1}z^{m-k}. \]
It is easy to represent these data with respect to the other basis
$(e_it^jz^k)_{0\le i,j<2,0\le k<m}$.

Note $\lambda(1)=1$, and that this $A$ is cosemisimple.
It is known that $A$ is semisimple if and only if $\operatorname{char}k\nmid 2m$.
Moreover, if $k$ contains a primitive $4m$-th root of 1,
then $A$ is selfdual, that is, $A\simeq A^*$ as Hopf algebras.

Table~\ref{tab:Kac} lists the invariant $v_A(H)$ for the handlebody-knots $0_1,\ldots,6_{16}$
in the table given in~\cite{IshiiKishimotoMoriuchiSuzuki12} when $m=3,\ldots,7$.

\begin{table}
\begin{center}
\begin{tabular}{|c||r|r|r|r|r|}
\hline
& $m=3$ & $m=4$ & $m=5$ & $m=6$ & $m=7$ \\
\hline
$0_1$ & $144$ & $256$ & $400$ & $576$ & $784$ \\
\hline
$4_1$ & $216$ & $256$ & $400$ & $864$ & $784$ \\
\hline
$5_1$ & $144$ & $256$ & $400$ & $576$ & $784$ \\
\hline
$5_2$ & $216$ & $256$ & $400$ & $864$ & $784$ \\
\hline
$5_3$ & $144$ & $256$ & $400$ & $576$ & $784$ \\
\hline
$5_4$ & $144$ & $256$ & $400$ & $576$ & $784$ \\
\hline
$6_1$ & $144$ & $256$ & $400$ & $576$ & $784$ \\
\hline
$6_2$ & $144$ & $256$ & $400$ & $576$ & $784$ \\
\hline
$6_3$ & $144$ & $256$ & $400$ & $576$ & $784$ \\
\hline
$6_4$ & $144$ & $256$ & $400$ & $576$ & $784$ \\
\hline
$6_5$ & $144$ & $256$ & $400$ & $576$ & $784$ \\
\hline
$6_6$ & $144$ & $256$ & $400$ & $576$ & $784$ \\
\hline
$6_7$ & $144$ & $256$ & $800$ & $576$ & $784$ \\
\hline
$6_8$ & $144$ & $256$ & $400$ & $576$ & $784$ \\
\hline
$6_9$ & $216$ & $256$ & $400$ & $864$ & $784$ \\
\hline
$6_{10}$ & $144$ & $256$ & $400$ & $576$ & $784$ \\
\hline
$6_{11}$ & $144$ & $256$ & $400$ & $576$ & $784$ \\
\hline
$6_{12}$ & $144$ & $256$ & $800$ & $576$ & $784$ \\
\hline
$6_{13}$ & $216$ & $256$ & $400$ & $864$ & $784$ \\
\hline
$6_{14}$ & $288$ & $256$ & $400$ & $1152$ & $784$ \\
\hline
$6_{15}$ & $288$ & $256$ & $400$ & $1152$ & $784$ \\
\hline
$6_{16}$ & $144$ & $256$ & $400$ & $576$ & $784$ \\
\hline
\end{tabular}
\end{center}
\caption{$\mathcal{B}_{4m}$}
%\caption{$\langle x,y,z\,|\,x^2=y^2=1,yx=xy,zx=xz,z^m=x,zy=yz^{-1}=xyz^{m-1}\rangle$}
\label{tab:Kac}
\end{table}
\end{example}

\section{The invariants derived from the finite quantum group $\overline{U}_q$}

Recall from \cite[Sect.VI.5]{Kassel95} the finite quantum group
$\overline{U}_q$ associated to $sl_2$.
Let $q\in k\setminus\{\pm1\}$ be a root of 1,
and let $e$ ($>1$) denote the order of $q^2$.
As an algebra, $\overline{U}_q$ is generated by $K$, $E$ and $F$,
and is defined by the relations
\begin{align*}
&KE=q^2EK, & &KF=q^{-2}FK, & &EF-FE=\frac{K-K^{-1}}{q-q^{-1}}, \\
&K^e=1, & &E^e=F^e=0.
\end{align*}
This $\overline{U}_q$ is a Hopf algebra with respect to the structure
\begin{align*}
\Delta(K)&=K\otimes K, & \varepsilon(K)&=1, & S(K)&=K^{-1}, \\
\Delta(E)&=1\otimes E+E\otimes K, & \varepsilon(E)&=0, & S(E)&=-EK^{-1},\\
\Delta(F)&=K^{-1}\otimes F+ F\otimes 1, & \varepsilon(F)&=0, & S(F)&=-K F.
\end{align*}

To apply results of Radford~\cite{Radford94b},
it is convenient to replace the generators above with
\begin{align*}
&a=K, & &x=\frac{1}{q-q^{-1}}FK, & &y=E.
\end{align*}
Then the defining relations turn into
\begin{align*}
&xa=q^2ax, & &ya=q^{-2}ay, & &yx-q^{-2}xy=a^2-1, & &a^e=1, & &x^e=y^e=0.
\end{align*}
We have
\begin{align*}
&\Delta(x)=1\otimes x+x\otimes a, &
&\varepsilon(x)=0, &
&S(x) = -xa^{-1}.
\end{align*}
The Hopf algebra $\overline{U}_q$ thus presented coincides with
Radford's $U_{(N, \nu, \omega)}$ in the special situation when $N=e$,
$\nu=1$ and $\omega=q^2$; see \cite[Sect.5.2]{Radford94b}.
We remark that the $q$ in \cite{Radford94b} should read our $q^2$.
The first three parts of the following proposition are proved in
Propositions~10, 11 of \cite{Radford94b}.

\begin{proposition}\label{prop:Uq}
\begin{itemize}
\item[(1)]
$(a^ix^jy^k)_{0\le i,j,k<e}$ is a basis of $\overline{U}_q$, so that
$\dim\overline{U}_q=e^3$.
\item[(2)]
$\overline{U}_q$ is unimodular, and
\[ \Lambda=\left(\sum_{i=0}^{e-1}a^i\right)x^{e-1}y^{e-1} \]
is a non-zero two-sided integral in $\overline{U}_q$.
\item[(3)]
The elements $\lambda$, $\lambda'$ of $\overline{U}_q^*$ defined by
\begin{align*}
&\lambda(a^ix^jy^k)=\delta_{(i,j,k),(0,e-1,e-1)}, &
&\lambda'(a^ix^jy^k)=\delta_{(i,j,k),(2,e-1,e-1)},
\end{align*}
where $0\le i,j,k<e$, are the left and the right integrals,
respectively, in $\overline{U}_q^*$ such that
$\lambda(\Lambda)=1=\lambda'(\Lambda)$.
It follows that $\overline{U}_q$ is not cosemisimple.
\item[(4)] 
$\overline{U}_q$ satisfies the assumption of Lemma~\ref{lem:double-horns}.
\end{itemize}
\end{proposition}

\begin{proof}
Let us prove Part 4.
Since $\lambda(\Lambda)=\lambda'(\Lambda)$ by Part 3,
we see from Proposition~\ref{prop:integral-integral} that $\lambda\circ S=\lambda'$.
Let $z\in Z(A)$.
Since $z$ commutes with $a$, we have
$z=\sum_{i,j=0}^{e-1}c_{ij}a^ix^jy^j$ with $c_{ij}\in k$.
To prove $\lambda(z)=\lambda'(z)$, we wish to show that
$c_{0,e-1}=c_{2,e-1}$.
The formula given in \cite[p.256, lines 2--3]{Radford94b} tells us that
for each $0<j<e$,
\begin{align}
yx^j=q^{-2j}x^jy+(j)_{q^2}a^2x^{j-1}-q^{-2(j-1)}(j)_{q^2}x^{j-1}, \label{eq:yx}
\end{align}
where $(j)_{q^2}=\sum_{t=0}^{j-1}q^{2t}$.
This implies that the term $a^2x^{e-2}y^{e-1}$ in $yz$ arises from the
product of $y$ with the terms $a^2x^{e-2}y^{e-2}$, $x^{e-1}y^{e-1}$,
$a^2x^{e-1}y^{e-1}$ in $z$.
It follows that the coefficient of $a^2x^{e-2}y^{e-1}$ in $yz$ equals
\[ c_{2,e-2}q^{-2e}+c_{0,e-1}(e-1)_{q^2}-c_{2,e-1}q^{-2e}(e-1)_{q^2}
=c_{2,e-2}+(c_{0,e-1}-c_{2,e-1})(e-1)_{q^2}, \]
while the same coefficient in $zy$ equals $c_{2,e-2}$.
This proves $c_{0,e-1}=c_{2,e-1}$, as desired.
\end{proof}

In what follows we suppose that the base field $k$ is the field
$\mathbb{C}$ of complex numbers.
Hence, $q^{-1}$ equals the complex conjugate $\overline{q}$ of $q$.
We re-choose $\Lambda$, $\lambda$ given in Proposition~\ref{prop:Uq}
(2), (3) so that the derived invariants behave preferably with mirror images.
For $q$ as above, we define complex numbers $c_q$, $\epsilon_q$ by
\begin{align*}
&c_q=\overline{q}^2(q-\overline{q})^{e-1}, &
&\epsilon_q = q^e. 
\end{align*}
Note that $\epsilon_q=1$ if the order $\operatorname{ord}q$ of $q$ is
odd, and $\epsilon_q=-1$ if $\operatorname{ord}q$ is even, and so that
$\epsilon_q=\epsilon_{\overline{q}}$.
We define
\begin{align*}
&\Lambda_q=c_q\Lambda, &
&\lambda_q=c_q^{-1}\lambda.
\end{align*}
One sees that $\Lambda_q$ is a two-sided integral in $\overline{U}_q$,
and $\lambda_q\in I_l(\overline{U}_q^*)$ with $\lambda_q(\Lambda_q)=1$.

\begin{lemma} 
We have
\[ \Lambda_q=\epsilon_q^{e-1}\left(\sum_{i=0}^{e-1}K^i\right)F^{e-1}E^{e-1}. \]
\end{lemma}

\begin{proof}
This follows since one computes
\begin{align*}
&\left(\sum_{i=0}^{e-1}K^i\right)F^{e-1}E^{e-1}
=(q-\overline{q})^{e-1}\left(\sum_{i=0}^{e-1}a^i\right)(xa^{-1})^{e-1}y^{e-1} \\
&=q^{-2\binom{e-1}{2}}(q-\overline{q})^{e-1}\Lambda
=q^{-e(e-1)}q^{-2}(q-\overline{q})^{e-1}\Lambda
=\epsilon_q^{e-1}c_q\Lambda
=\epsilon_q^{e-1}\Lambda_q.
\end{align*}
\end{proof}

\begin{lemma}\label{lem:Uq-isom}
$K \mapsto K^{op}$, $E\mapsto E^{op}$ and $F\mapsto F^{op}$ give a Hopf
algebra isomorphism
$\overline{U}_{\overline{q}}\overset{\simeq}{\longrightarrow}\overline{U}_q^{op}$,
under which $\Lambda_{\overline{q}}\mapsto\Lambda_{q}^{op}$.
\end{lemma}

\begin{proof}
It is well-known that the correspondences above gives a Hopf algebra
isomorphism.
To see that $\Lambda_{\overline{q}}\mapsto\Lambda_{q}^{op}$, it suffices
to prove that
\[ E^{e-1}F^{e-1}\left(\sum_{i=0}^{e-1}K^i\right)
=\left(\sum_{i=0}^{e-1}K^i\right)F^{e-1}E^{e-1} \]
in $\overline{U}_q$, since $\epsilon_{\overline{q}}=\epsilon_q$.
By \eqref{eq:yx} for $j=e-1$, we see
\[ y^{e-1}x^{e-1}\left(\sum_{i=0}^{e-1}a^i\right)
=\overline{q}^2\left(\sum_{i=0}^{e-1}a^i\right)x^{e-1}y^{e-1}. \]
By multiplying $a^{-(e-1)}=a$ from the right, it follows that
\[ y^{e-1}(xa^{-1})^{e-1}\left(\sum_{i=0}^{e-1}a^i\right)
=\left(\sum_{i=0}^{e-1}a^i\right)(x a^{-1})^{e-1}y^{e-1}, \]
which implies the desired equality.
\end{proof}

By Proposition~\ref{prop:Uq} (4), $\overline{U}_q$ together with
$\lambda_q$ defines the invariant $v_{\overline{U}_q}(H)$ for each
handlebody-knot $H$.
Let us write simply $v_q(H)$ for this.

\begin{proposition}
Given a handlebody-knot $H$, the invariant $v_q(H^*)$ of the
mirror image $H^*$ of $H$ equals the complex conjugate
$\overline{v_q(H)}$ of $v_q(H)$, that is, $v_q(H^*)=\overline{v_q(H)}$.
\end{proposition}

\begin{proof}
By Proposition~\ref{prop:integral-integral}, the composite of the
isomorphism in Lemma~\ref{lem:Uq-isom} with $\lambda_q$ coincides with
$\lambda_{\overline{q}}$.
Then Proposition~\ref{prop:mirror-image-knot} shows
$v_q(H^*)=v_{\overline{q}}(H)$.
It remains to prove $v_{\overline{q}}(H)= \overline{v_q(H)}$.
This equality holds since the Hopf algebra $\overline{U}_{\overline{q}}$
and the linear map
$\lambda_{\overline{q}}:\overline{U}_{\overline{q}}\to\mathbb{C}$ are
the base extensions of $\overline{U}_q$, $\lambda_q$, respectively,
along the complex conjugation $\mathbb{C}\to\mathbb{C}$.
\end{proof}

\begin{remark}
Let $q=e^{2\pi\sqrt{-1}/n}$.
For $n\leq4$ we have checked by computer calculation
that the invariant does not detect the handlebody-knots $0_1,\ldots,6_{16}$
given in~\cite{IshiiKishimotoMoriuchiSuzuki12}.
For $n>4$ the calculation takes so far too long time for us
to see whether the invariant is non-trivial.
\end{remark}

%\section*{Acknowledgment}

\end{document}